\documentclass[11pt,a4paper,dvipsnames,x11names]{amsart}
\usepackage[margin=30truemm]{geometry}
\usepackage{stix2}
\usepackage{microtype}
\usepackage{amsmath,amsthm,mathtools,amscd}
\usepackage{aliascnt}
\usepackage{enumitem}
\usepackage{xcolor}
\usepackage{tikz}
\usepackage{tikz-cd}
\usetikzlibrary{positioning,arrows.meta,decorations.markings,calc}

\usepackage[backend=biber,style=alphabetic,sorting=nyt,maxnames=20]{biblatex}
\usepackage{xurl}
\usepackage{hyperref}
\usepackage{bookmark}
\usepackage{cleveref}

\definecolor{FushimiBlue}{HTML}{175A8C}
\definecolor{FushimiRed}{HTML}{8B2F3F}
\definecolor{FushimiViolet}{HTML}{66507A}
\definecolor{FushimiOrange}{HTML}{A25B23}

\hypersetup{
	colorlinks=true,
	linkcolor=FushimiRed,
	citecolor=FushimiOrange,
	urlcolor=FushimiOrange,
	filecolor=FushimiBlue,
	pdfauthor={Riku Fushimi},
	pdftitle={Cluster Morita theorem for negative cluster categories}
}

\allowdisplaybreaks
\setlist[enumerate]{leftmargin=2.35em,itemsep=0.12em,topsep=0.32em,parsep=0pt}
\setlist[itemize]{leftmargin=2.15em,itemsep=0.12em,topsep=0.32em,parsep=0pt}

\newtheoremstyle{fushimi}
	{0.65\baselineskip}
	{0.55\baselineskip}
	{\normalfont}
	{}
	{\bfseries}
	{.}
	{0.5em}
	{}
\theoremstyle{fushimi}
\newtheorem{mainthm}{Theorem}

\newtheorem{thm}{Theorem}[section]

\newaliascnt{prop}{thm}
\newtheorem{prop}[prop]{Proposition}
\aliascntresetthe{prop}
\newaliascnt{lem}{thm}
\newtheorem{lem}[lem]{Lemma}
\aliascntresetthe{lem}
\newaliascnt{cor}{thm}
\newtheorem{cor}[cor]{Corollary}
\aliascntresetthe{cor}
\newaliascnt{dfn}{thm}
\newtheorem{dfn}[dfn]{Definition}
\aliascntresetthe{dfn}
\newaliascnt{rmk}{thm}
\newtheorem{rmk}[rmk]{Remark}
\aliascntresetthe{rmk}
\newaliascnt{exa}{thm}
\newtheorem{exa}[exa]{Example}
\aliascntresetthe{exa}

\numberwithin{equation}{section}

\crefname{mainthm}{Theorem}{Theorems}
\Crefname{mainthm}{Theorem}{Theorems}
\crefname{thm}{Theorem}{Theorems}
\crefname{prop}{Proposition}{Propositions}
\crefname{lem}{Lemma}{Lemmas}
\crefname{cor}{Corollary}{Corollaries}
\crefname{dfn}{Definition}{Definitions}
\crefname{rmk}{Remark}{Remarks}
\crefname{exa}{Example}{Examples}
\crefname{section}{Section}{Sections}
\crefname{subsection}{Subsection}{Subsections}

\newcommand{\defterm}[1]{\textcolor{blue}{\textit{#1}}}

\DeclareMathOperator{\Hom}{Hom}
\DeclareMathOperator{\End}{End}
\DeclareMathOperator{\Ext}{Ext}
\DeclareMathOperator{\id}{id}
\DeclareMathOperator{\Filt}{Filt}
\DeclareMathOperator{\add}{add}
\DeclareMathOperator{\thick}{thick}
\DeclareMathOperator{\per}{per}
\DeclareMathOperator{\pvd}{pvd}
\DeclareMathOperator{\cosg}{cosg}
\DeclareMathOperator{\sg}{sg}
\DeclareMathOperator{\CM}{CM}
\DeclareMathOperator{\modu}{mod}
\DeclareMathOperator{\REnd}{\mathbf{R}End}
\DeclareMathOperator{\RHom}{\mathbf{R}Hom}
\DeclareMathOperator{\HHom}{HH}
\DeclareMathOperator{\HChom}{HC}
\DeclareMathOperator{\HN}{HN}

\newcommand{\A}{\mathcal A}
\newcommand{\C}{\mathcal C}
\newcommand{\D}{\mathcal D}
\newcommand{\Hcal}{\mathcal H}
\newcommand{\Lcal}{\mathcal L}
\newcommand{\F}{\mathcal F}
\newcommand{\DD}{\mathbb D}
\newcommand{\Serre}{\mathbb S}
\newcommand{\Xcal}{\mathcal X}
\newcommand{\Ycal}{\mathcal Y}
\newcommand{\Scal}{\mathcal S}
\newcommand{\Pdg}{\mathcal P}
\newcommand{\Tdg}{\mathcal T}
\newcommand{\Qdg}{\mathcal Q}

\newcommand{\dg}{\mathrm{dg}}
\newcommand{\xto}[1]{\xrightarrow{#1}}
\newcommand{\mono}{\hookrightarrow}

\title[Cluster Morita theorem for negative cluster categories]{Cluster Morita theorem for negative cluster categories}
\author{Riku Fushimi}
\address{Department of Mathematics, Nagoya University, Chikusa-ku, Nagoya 464-8602, Japan}
\email{fushimi.riku.h9@s.mail.nagoya-u.ac.jp}
\date{}

\begin{document}

\begin{abstract}
Fix an integer $d\leq-2$. We characterize Hom-finite algebraic triangulated categories admitting a $(-d)$-simple-minded system as stable categories $\underline{\CM}(B)$ of proper $(-d)$-self-injective non-positive dg algebras; equivalently, each admits a $d$-stable locally finite strictly positive dg model whose cosingular dg quotient recovers the chosen enhancement. Under a $d$-Calabi--Yau hypothesis, we give a characterization theorem for acyclic negative cluster categories in terms of the finite graded extension algebra of a simple-minded system. At chain level, Hochschild and reduced cyclic localization identify right $(d+1)$-Calabi--Yau structures on the finite-dimensional part with normalized right $d$-Calabi--Yau structures on the cosingular quotient. Finally, when the Koszul dual is proper, the Brav--Dyckerhoff evaluation morphism is a quasi-isomorphism of mixed complexes, yielding left--right Calabi--Yau symmetry.
\end{abstract}

\maketitle

\setcounter{tocdepth}{1}
\tableofcontents

\section{Introduction}

Cluster categories were introduced by Buan--Marsh--Reineke--Reiten--Todorov and endowed with canonical triangulated structures by Keller \cite{BMRRT06,Keller05}. One of the basic themes of the subject is that distinguished finite data, most notably a cluster-tilting object, can control the ambient triangulated category. Keller--Reiten made this precise for algebraic $2$-Calabi--Yau categories with an acyclic cluster-tilting object, with higher-dimensional variants under additional vanishing hypotheses \cite{KR08}. This point of view has developed in several complementary directions: Amiot and Guo constructed generalized and higher cluster categories, Iyama--Yang developed a general framework realizing Verdier quotients as ideal quotients, including the fundamental-domain realization of Amiot--Guo--Keller cluster categories, and Hanihara, Keller--Liu, and Tomonaga obtained Morita- or recognition-type results in related settings \cite{Amiot09,Guo11,IY20,Hanihara22,KellerLiuAmiot24,Tomonaga26}.

On the negative Calabi--Yau side, the corresponding distinguished objects are simple-minded systems. They were introduced in stable module categories by Koenig--Liu and developed in general triangulated categories by Dugas \cite{KoenigLiu12,Dugas15}. Higher simple-minded systems and negative cluster categories have since been studied through mutation, reduction, silting theory, simple-minded collections, quotient constructions, and Coxeter combinatorics \cite{CoelhoPauksztello20,Jin23Reduction,IyamaJin23,CoelhoPauksztelloPloog22,Fedele22}. Jin also introduced $w$-self-injective dg algebras and proved that their simple modules form $w$-simple-minded systems in stable categories of Cohen--Macaulay dg modules \cite[Theorem~5.6]{Jin20}. Thus there is a well-developed theory producing and studying simple-minded systems in negative Calabi--Yau and singularity settings. What has been missing, to the best of our knowledge, is a general Morita-type theorem in the reverse direction: starting from an abstract algebraic triangulated category equipped with a simple-minded system, reconstruct the ambient category from the simple-minded data.

There is an immediate obstruction to using ordinary endomorphisms for this purpose. If $\Lcal=\{L_1,\dots,L_r\}$ is a basic simple-minded system and $L=\bigoplus_iL_i$, then
\[
\End(L)\simeq k^r.
\]
Hence degree-zero endomorphisms remember essentially only the number of simple-minded objects. The ambient information begins in higher extensions, and the natural invariant is the dg endomorphism algebra $\REnd(L)$. This is precisely where positive/non-positive Koszul duality enters. In \cite[Theorem~D]{Fus24}, contravariant Koszul duality for locally finite dg algebras was developed so as to exchange perfect and finite-dimensional derived categories. The basic idea of this paper is to use the positive side of that duality as a Morita model and to recover the original category as a cosingular quotient.

Fix $d\leq-2$. For a dg algebra $A$ with $\pvd(A)\subseteq\per(A)$, write
\[
\cosg(A)=\per(A)/\pvd(A).
\]
Our first result is a dg-level realization theorem: every Hom-finite algebraic triangulated category with a $(-d)$-simple-minded system admits a strictly positive model whose cosingular quotient recovers the chosen enhancement.

\begin{mainthm}[Cosingular realization; \cref{thm:realization}]\label{main:realization}
Let $\C$ be a Hom-finite algebraic triangulated category equipped with a pretriangulated dg enhancement $\C_{\dg}$, and let $\Lcal=\{L_1,\dots,L_r\}$ be a $(-d)$-simple-minded system in $\C$. Then there exists a locally finite strictly positive dg algebra $A$ with $A^0\simeq k^r$ and an exact quasi-functor $\per_{\dg}(A)\to\C_{\dg}$ inducing a quasi-equivalence
\[
\cosg_{\dg}(A)=\per_{\dg}(A)/\pvd_{\dg}(A)\simeq\C_{\dg}.
\]
Moreover, $A$ is $d$-stable in the sense of \cref{dfn:d-stable}.
\end{mainthm}

The theorem has an equivalent Cohen--Macaulay interpretation. Let $B=A^!=\REnd_A(A^0)$ be the Koszul dual. The $d$-stability of $A$ is equivalent to $B$ being $(-d)$-self-injective, that is,
\[
\add\DD B=\add\Sigma^{d+1}B.
\]
In the locally finite setting considered here, this equality forces $B$ to be proper. Positive/non-positive Koszul duality and Jin's Cohen--Macaulay duality therefore give $\C\simeq\underline{\CM}(B)$. Together with Jin's converse construction \cite[Theorem~5.6]{Jin20}, this yields the characterization
\[
\C\text{ admits a $(-d)$-simple-minded system}
\quad\Longleftrightarrow\quad
\C\simeq\underline{\CM}(B)
\]
for some locally finite $(-d)$-self-injective non-positive dg algebra $B$; see \cref{cor:CM-characterization}.

The Cohen--Macaulay interpretation above is summarized by the diagram
\[
\begin{tikzcd}[column sep=large,row sep=large]
\pvd_{\dg}(A) \arrow[r,hook] \arrow[d,"\sim"',sloped] &
\per_{\dg}(A) \arrow[r] \arrow[d,"\sim"',sloped] &
\cosg_{\dg}(A) \arrow[d,"\sim",sloped] \\
\per_{\dg}(B^{\mathrm{o}})^{\mathrm{o}} \arrow[r,hook] &
\pvd_{\dg}(B^{\mathrm{o}})^{\mathrm{o}} \arrow[r] &
\sg_{\dg}(B^{\mathrm{o}})^{\mathrm{o}}.
\end{tikzcd}
\]
The two left vertical arrows are the Koszul-duality equivalences, and hence induce the right vertical equivalence on the quotients. For the realizing algebra, the upper-right term is $\C_{\dg}$, while Jin's Cohen--Macaulay description and duality identify the homotopy category of the lower-right term with $\underline{\CM}(B)$. Thus the diagram explains the identification $\C\simeq\underline{\CM}(B)$ associated with \cref{main:realization}. As a further consequence, properness of $B$ forces the positive model $A$ to be homologically smooth and $\pvd_{\dg}(A)$ to be proper; see \cref{prop:smooth-proper}.

The construction itself is controlled by a finite extension window. Writing $\Hcal_{\Lcal}=\Filt(\Lcal)$, the simple-minded-system axioms give
\[
\C=\Hcal_{\Lcal}*\Sigma^{-1}\Hcal_{\Lcal}*\cdots*\Sigma^{d+1}\Hcal_{\Lcal}.
\]
We extract from $\REnd_{\C_{\dg}}(L)$ a strictly positive dg algebra retaining exactly the cohomology relevant to this window, and show that the corresponding positive window is a fundamental domain for the cosingular quotient. Under Koszul duality, the $d$-stability of this model becomes the shifted self-injectivity of $B$.

For a finite acyclic quiver $Q$, let $S_Q$ be the direct sum of the simple right $kQ$-modules and put
\[
kQ^!:=\Ext^*_{kQ}(S_Q,S_Q),
\qquad
\C_d(kQ)=\D^b(kQ)/(\Sigma^{-d}\Serre),
\]
where $kQ^!$ is graded by extension degree and carries the Yoneda product. Since $kQ$ is hereditary, $kQ^!$ is concentrated in degrees zero and one, and its degree-one extension quiver is $Q$.
On the cluster-tilting side, Keller--Reiten recognition starts from the endomorphism algebra of the distinguished object. Since the degree-zero endomorphism algebra of a simple-minded system is always semisimple, the negative counterpart must instead use higher extensions; \cref{main:negative-KR} shows that a finite range suffices.

\begin{mainthm}[Negative Keller--Reiten recognition theorem; \cref{thm:negative-KR}]\label{main:negative-KR}
Let $d\leq-3$, and let $\C$ be a Hom-finite algebraic $d$-Calabi--Yau triangulated category with a $(-d)$-simple-minded system $\Lcal$, with $L=\bigoplus_{X\in\Lcal}X$. Let $Q$ be a finite acyclic quiver. If
\[
\bigoplus_{i=0}^{-d-1}\Hom_\C(L,\Sigma^iL)\simeq kQ^!
\]
as graded algebras, then
\[
\C\simeq\C_d(kQ).
\]
The same conclusion holds for $d=-2$ provided $\Hcal_\Lcal$ is hereditary; see \cref{rmk:d-minus-2}.
\end{mainthm}

Because $kQ^!$ is concentrated in degrees zero and one, the hypothesis has a particularly concrete form: the extension quiver of $\Hcal_{\Lcal}$ is $Q$ and
\[
\Hom_\C(L,\Sigma^iL)=0\qquad(2\leq i\leq-d-1).
\]
There is no further multiplicative condition in this range. The standard simple-minded system of $\C_d(kQ)$ satisfies these assumptions. Consequently, for $d\leq-3$,
\[
\C\simeq\C_d(kQ)
\quad\Longleftrightarrow\quad
\C\text{ admits such a $(-d)$-simple-minded system};
\]
see \cref{cor:negative-KR-characterization}. The proof of \cref{main:negative-KR} identifies the Koszul dual cohomology with
\[
H^*(B)\simeq kQ\ltimes\Sigma^{-d-1}\DD(kQ),
\]
and then uses intrinsic formality and Jin's singularity-category description to recover $\C_d(kQ)$. \Cref{exa:cyclic-family} shows that the $d$-stability in \cref{main:realization} is strictly weaker than requiring the cosingular quotient to be $d$-Calabi--Yau.

The second theme of the paper is Calabi--Yau structure at the dg level. A triangulated Calabi--Yau property records a Serre-duality isomorphism, whereas a dg Calabi--Yau structure is the stronger datum of a non-degenerate cyclic class lifting the corresponding bimodule duality. For the positive model $A$, two operations are relevant: localization from the finite-dimensional part to the cosingular quotient, and Koszul/Morita duality between the positive model and its finite-dimensional part.

Assume first that $A$ is $d$-stable, and put $\ell=A^0\simeq k^r$. Since $\HChom_*(\ell)$ contributes a semisimple periodic summand, we remove this contribution by taking the cones of the maps $M(\ell)\to M(-)$. This leads to $\ell$-normalized right Calabi--Yau structures on $\cosg_{\dg}(A)$.

\begin{mainthm}[Hochschild localization and normalized Calabi--Yau correspondence; \cref{thm:normalized-CY}]\label{main:normalized-CY}
Let $d\leq-2$, and let $A$ be a $d$-stable locally finite strictly positive dg algebra with $\ell=A^0\simeq k^r$. Then the Hochschild localization boundary is an isomorphism
\[
\partial_{\mathrm{HH}}:
\HHom_{-d}(\cosg_{\dg}(A))
\xrightarrow{\sim}
\HHom_{-d-1}(\pvd_{\dg}(A)),
\]
and its dual preserves and detects non-degenerate Hochschild classes. Moreover, the dual of the reduced cyclic localization boundary induces a bijection
\[
\left\{\begin{array}{c}\text{right $(d+1)$-Calabi--Yau structures}\\ \text{on $\pvd_{\dg}(A)$}\end{array}\right\}
\xrightarrow{\sim}
\left\{\begin{array}{c}\text{$\ell$-normalized right $d$-Calabi--Yau structures}\\ \text{on $\cosg_{\dg}(A)$}\end{array}\right\}.
\]
\end{mainthm}

The main point is to show that the dual Hochschild boundary preserves and detects non-degeneracy; we prove this using local $\pvd(A)$-envelopes, the Keller--Nicol\'as weight structure, and the localization formalism of Hanihara--Liu \cite{KN13,HaniharaLiu26}. We also determine the relation with ordinary right Calabi--Yau structures: for odd $d$ every ordinary structure admits a normalized lift, while for even $d$ the normalized classes are exactly those annihilating the semisimple contribution; see \cref{prop:parity}.

The second operation is left--right duality. Brav--Dyckerhoff construct an evaluation morphism from the mixed complex of a smooth dg category to the dual mixed complex of a locally perfect Morita dual \cite{BravDyckerhoff19}; related Calabi--Yau dualities appear in \cite{KellerLiuAmiot24,HolsteinRivera24}. In the ordinary positive dg setting of this paper, properness of the Koszul dual is enough for this evaluation morphism to be a quasi-isomorphism of mixed complexes.

\begin{mainthm}[Mixed-complex duality and left--right Calabi--Yau symmetry; \cref{thm:left-right-CY}]\label{main:left-right-CY}
Let $A$ be a locally finite strictly positive dg algebra with $A^0\simeq k^r$ and proper Koszul dual $B=A^!$. Then the Brav--Dyckerhoff evaluation morphism
\[
M\bigl(\per_{\dg}(A)\bigr)\xrightarrow{\sim}\DD M\bigl(\pvd_{\dg}(A)\bigr)
\]
is a quasi-isomorphism of mixed complexes. Consequently, for every $n\in\mathbb Z$ it induces canonical isomorphisms
\[
\HHom_n\bigl(\per_{\dg}(A)\bigr)\xrightarrow{\sim}\DD\HHom_{-n}\bigl(\pvd_{\dg}(A)\bigr),
\qquad
\HN_n\bigl(\per_{\dg}(A)\bigr)\xrightarrow{\sim}\DD\HChom_{-n}\bigl(\pvd_{\dg}(A)\bigr).
\]
The second isomorphism identifies left $n$-Calabi--Yau structures on $\per_{\dg}(A)$ with right $n$-Calabi--Yau structures on $\pvd_{\dg}(A)$. If $A$ is moreover $d$-stable, then for $n=d+1$ \cref{main:normalized-CY} yields canonical bijections
\[
\begin{gathered}
\left\{\begin{array}{c}\text{left $(d+1)$-Calabi--Yau structures}\\ \text{on $\per_{\dg}(A)$}\end{array}\right\}
\xrightarrow{\sim}
\left\{\begin{array}{c}\text{right $(d+1)$-Calabi--Yau structures}\\ \text{on $\pvd_{\dg}(A)$}\end{array}\right\}\\[1mm]
\xrightarrow{\sim}
\left\{\begin{array}{c}\text{$\ell$-normalized right $d$-Calabi--Yau structures}\\ \text{on $\cosg_{\dg}(A)$}\end{array}\right\}.
\end{gathered}
\]
\end{mainthm}

Taken together, the four results are organized by two recurring operations: passage to the cosingular quotient and passage across Koszul duality. At the triangulated level they underlie the realization and recognition results, while at chain level they transport Calabi--Yau structures across $\pvd_{\dg}(A)$, $\per_{\dg}(A)$, and $\cosg_{\dg}(A)$.

\subsection*{Organization and conventions}

In \cref{sec:sms-positive}, we recall simple-minded systems and the Koszul duality used throughout the paper, while \cref{sec:cosingular-fundamental-domain} develops $d$-stable positive dg algebras and the cosingular fundamental domain. \Cref{sec:cosingular-realization,sec:acyclic-negative-CY} prove \cref{main:realization,main:negative-KR}, respectively. \Cref{sec:cy,sec:left-right-CY} treat Calabi--Yau structures and prove \cref{main:normalized-CY,main:left-right-CY}.

Throughout, $k$ is an algebraically closed field of characteristic zero. All triangulated categories are $k$-linear and Hom-finite unless otherwise stated, all dg categories are $k$-linear and cohomologically graded, and all modules are right modules. A superscript ${}^{\mathrm{o}}$ denotes the opposite algebra or category. We write $\DD=\Hom_k(-,k)$ for the $k$-linear dual. A dg algebra $A$ is called \defterm{locally finite} if $\dim_k H^i(A)<\infty$ for every $i\in\mathbb Z$, and \defterm{proper} if $\dim_k H^*(A)<\infty$. For a dg algebra $A$, we set
\[
\per(A)=\thick_{\D(A)}(A),
\qquad
\pvd(A)=\{X\in\D(A)\mid \dim_k H^*(X)<\infty\}.
\]
We call $A$ \defterm{pvd-finite} if $\pvd(A)$ is Hom-finite. A dg algebra $A$ is \defterm{homologically smooth} if the diagonal bimodule belongs to $\per(A^e)$; equivalently, $\per_{\dg}(A)$ is a smooth dg category. A small dg category is \defterm{proper} if each of its morphism complexes has finite-dimensional total cohomology. Whenever $\pvd(A)\subseteq\per(A)$, we write $\cosg(A)=\per(A)/\pvd(A)$ and $\cosg_{\dg}(A)=\per_{\dg}(A)/\pvd_{\dg}(A)$; whenever $\per(A)\subseteq\pvd(A)$, we write $\sg(A)=\pvd(A)/\per(A)$ and $\sg_{\dg}(A)=\pvd_{\dg}(A)/\per_{\dg}(A)$. For full subcategories $\Xcal,\Ycal$ of a triangulated category, $\Xcal*\Ycal$ denotes the ordinary extension product: its objects are the middle terms of triangles $X\to Z\to Y\to\Sigma X$ with $X\in\Xcal$ and $Y\in\Ycal$. The notation $\Filt(\Scal)$ denotes the union of the finite iterated extension products of $\Scal$; no Karoubi closure is built into $*$ or $\Filt$.

\par\medskip
\noindent
{\bf Acknowledgements.}
The author would like to express his sincere gratitude to his supervisor Akira Ishii for his continuous support and encouragement.
The author is grateful to Ryu Tomonaga for many valuable comments, suggestions, and helpful advice concerning this work.

\medskip
\noindent
{\bf Use of artificial intelligence.}
The author used ChatGPT (OpenAI) and Claude Opus (Anthropic) during the preparation of this manuscript for assistance with language editing, exposition, and manuscript-level consistency checks. The author takes full responsibility for all mathematical statements and arguments in the paper.

\section{Simple-minded systems and positive dg algebras}\label{sec:sms-positive}

\subsection{Simple-minded systems}

For a full subcategory $\Xcal$ and an integer $j\leq0$, set
\[
\F^0(\Xcal)=\Xcal,
\qquad
\F^j(\Xcal)=\Xcal*\Sigma^{-1}\Xcal*\cdots*\Sigma^j\Xcal\quad(j<0).
\]

\begin{dfn}\label{dfn:sms}
Let $d\leq-2$ and let $\C$ be a $k$-linear triangulated category. A finite collection $\Lcal=\{L_1,\dots,L_r\}$ is called an \defterm{$(-d)$-simple-minded system} if, with $L=\bigoplus_iL_i$ and $\Hcal_{\Lcal}=\Filt(\Lcal)$, the following conditions hold:
\begin{itemize}
	\item [(i)] $\Hom_{\C}(L_i,L_j)=\delta_{ij}k$ for all $i,j$;
	\item [(ii)] $\Hom_{\C}(L,\Sigma^iL)=0$ for every $d+1\leq i\leq-1$;
	\item [(iii)] $\C=\F^{d+1}(\Hcal_{\Lcal})$.
\end{itemize}
\end{dfn}

This is the extension-window form of the usual $w$-simple-minded-system convention with $w=-d$; compare \cite[Definition~5.2]{Jin20} and \cite{Dugas15,CoelhoPauksztello20}. By \cref{dfn:sms}(i) and \cite[Lemma~2.7]{Dugas15}, $\Hcal_{\Lcal}=\Filt(\Lcal)$ is closed under direct summands.

\begin{lem}\label{lem:sms-window-summands}
For every $d+1\leq j\leq0$, the extension window $\F^j(\Hcal_{\Lcal})$ is closed under direct summands and finite direct sums.
\end{lem}
\begin{proof}
We argue by downward induction on $j$. The case $j=0$ is the summand-closure of $\Hcal_{\Lcal}$ recalled above. Suppose $j<0$ and the assertion is known for $j+1$. By \cref{dfn:sms}(ii) and finite d\'evissage,
\[
\Hom_\C\bigl(\F^{j+1}(\Hcal_{\Lcal}),\Sigma^j\Hcal_{\Lcal}\bigr)=0.
\]
We use \cite[Proposition~2.1(1)]{IyamaYoshino08} in the following form: if two full subcategories are closed under direct summands and $\Hom(\Xcal,\Ycal)=0$, then $\Xcal*\Ycal$ is closed under direct summands. The induction hypothesis and the summand-closure of $\Sigma^j\Hcal_{\Lcal}$ therefore show that $\F^j(\Hcal_{\Lcal})$ is summand-closed. Closure under finite direct sums is immediate from the additivity of the factors and finite direct sums of triangles.
\end{proof}

\begin{rmk}[Idempotent completeness]\label{rmk:idempotent-complete}
The proof of \cref{lem:sms-window-summands} uses only \cref{dfn:sms}(i),(ii). Hence it also applies to a pre-$(-d)$-simple-minded system, meaning a collection satisfying these two conditions. Let $\C^\pi$ be the idempotent completion of $\C$. The collection $\Lcal$ remains a pre-$(-d)$-simple-minded system in $\C^\pi$, and extension products of objects of $\C$ computed in $\C^\pi$ still belong to $\C$: indeed, the connecting morphism lies in $\C$, and its cone can be taken in $\C$. Hence, if $\Lcal$ is a $(-d)$-simple-minded system in $\C$, then $\F^{d+1}(\Hcal_{\Lcal})=\C$ is closed under direct summands in $\C^\pi$ by \cref{lem:sms-window-summands}. Every object of $\C^\pi$ is a direct summand of an object of $\C$, so $\C^\pi=\C$. Thus a triangulated category admitting a $(-d)$-simple-minded system is automatically idempotent complete.
\end{rmk}

\begin{lem}\label{lem:injectives}
The category $\Hcal_{\Lcal}$ is a length abelian category with enough injectives. Let $L\mono I$ be the minimal injective envelope of $L$. Then
\[
I\in\Hcal_{\Lcal}\cap\bigl(\Sigma^{-2}\Hcal_{\Lcal}*\Sigma^{-3}\Hcal_{\Lcal}*\cdots*\Sigma^{d}\Hcal_{\Lcal}\bigr)
\]
and
\[
\Hom_{\C}(L,I)\simeq k^r
\]
as a right $\End_{\C}(L)\simeq k^r$-module.
\end{lem}
\begin{proof}
By \cref{dfn:sms}(ii) and induction on extension lengths, we have $\Hom_\C(X,\Sigma^{-1}Y)=0$ for every $X,Y\in\Hcal_\Lcal$, and hence $\Hcal_\Lcal$ inherits an exact structure from $\C$. Since $\Lcal$ is a semibrick and $\Hcal_{\Lcal}=\Filt(\Lcal)$, Enomoto's theorem \cite{Enomoto21} shows that this exact category is a length abelian category with simple objects $L_1,\dots,L_r$.

Apply \cref{dfn:sms}(iii) to $\Sigma L$. After rotating and shifting the resulting triangle, we obtain
\[
L\longrightarrow J\longrightarrow H\longrightarrow\Sigma L
\]
with $H\in\Hcal_\Lcal$ and
\[
J\in\Sigma^{-2}\Hcal_{\Lcal}*\Sigma^{-3}\Hcal_{\Lcal}*\cdots*\Sigma^{d}\Hcal_{\Lcal}.
\]
Since $L,H\in\Hcal_\Lcal$, also $J\in\Hcal_\Lcal$. Moreover, \cref{dfn:sms}(ii), together with induction on extension factors, gives $\Hom_\C(\Hcal_\Lcal,\Sigma J)=0$, so $J$ is injective in $\Hcal_\Lcal$. The triangle above is a short exact sequence in the abelian category $\Hcal_\Lcal$, so $L\to J$ is a monomorphism. Since every simple object $L_i$ embeds into the injective object $J$, induction on length shows that every object of $\Hcal_\Lcal$ embeds into a finite direct sum of copies of $J$. Hence $\Hcal_\Lcal$ has enough injectives.

Let $L\mono I$ be the minimal injective envelope of $L$. Since $J$ is injective, the monomorphism $L\to J$ extends along $L\mono I$ to a morphism $I\to J$. Its kernel has zero intersection with the essential subobject $L\subseteq I$, hence it is zero; thus $I\to J$ is a monomorphism. Since $I$ is injective, it splits, so $I$ is a direct summand of $J$. By \cref{lem:sms-window-summands}, the window containing $J$ is closed under direct summands, and therefore $I$ belongs to the same window. Finally, the socle of $I$ is $L$, and therefore $\Hom_\C(L,I)\simeq k^r$ as a right $\End_\C(L)\simeq k^r$-module.
\end{proof}

\begin{rmk}
Let $\Lcal$ be the standard $(-d)$-simple-minded system of the negative cluster category $\C_d(kQ)=\D^b(kQ)/(\Sigma^{-d}\Serre)$ of an acyclic quiver $Q$. Then
\[
\add I=\Hcal_\Lcal\cap\Sigma^d\Hcal_\Lcal.
\]
Indeed, under the standard identification $\Hcal_\Lcal\simeq\modu kQ$, the injective objects are $\Serre P$ with $P\in\operatorname{proj}kQ$. Since $\Sigma^{-d}\Serre$ becomes isomorphic to the identity in the orbit category, we have $\Serre P\simeq\Sigma^dP$ in $\C_d(kQ)$. Hence $I\in\Sigma^d\Hcal_\Lcal$. Conversely, if $X=\Sigma^dY\in\Hcal_\Lcal\cap\Sigma^d\Hcal_\Lcal$, then $\Hom_\C(\Hcal_\Lcal,\Sigma X)=\Hom_\C(\Hcal_\Lcal,\Sigma^{d+1}Y)=0$, so $X$ is injective. The case $Q=A_3$ and $d=-3$ is illustrated in \Cref{fig:A3-negative-cluster}.
\end{rmk}

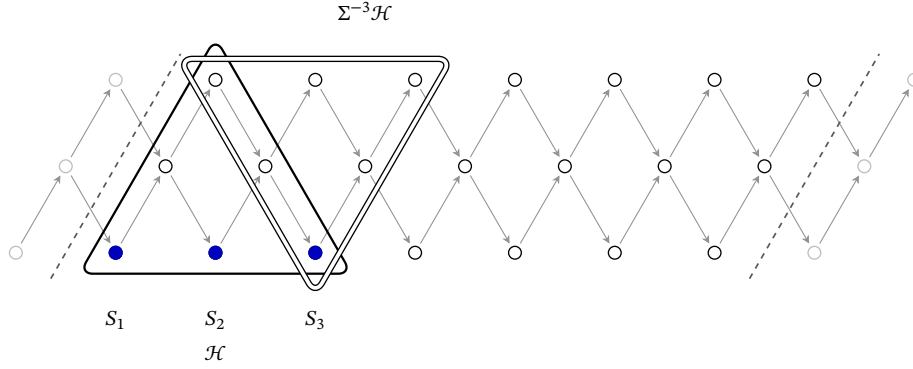
\begin{figure}[htbp]
\centering
\begin{tikzpicture}[
	x=.66cm,y=.66cm,
	ar/.style={-{Stealth[length=3pt,width=3pt]},draw=black!43,line width=.4pt,shorten <=4pt,shorten >=4pt},
	vertex/.style={circle,draw=black,fill=white,inner sep=0pt,minimum size=4.8pt,line width=.5pt},
	simple/.style={vertex,draw=blue!75!black,fill=blue!75!black},
	cut/.style={black!65,dash pattern=on 2.2pt off 2.2pt,line width=.65pt},
	every node/.style={font=\scriptsize}
]
	\pgfmathsetmacro{\h}{sqrt(3)}
	\def\m{.42}

	\begin{scope}
		\clip (-2.2,-.95) rectangle (17,{2*\h+1});

		\foreach \r in {0,1,2}{
			\foreach \j in {-1,...,7}{
				\pgfmathtruncatemacro{\u}{\r+2*\j}
				\coordinate (v\r-\j) at (\u,{\r*\h});
			}
		}

		\foreach \r in {0,1}{
			\pgfmathtruncatemacro{\s}{\r+1}
			\foreach \j in {-1,...,7}{
				\draw[ar] (v\r-\j) -- (v\s-\j);
			}
			\foreach \j in {-1,...,6}{
				\pgfmathtruncatemacro{\k}{\j+1}
				\draw[ar] (v\s-\j) -- (v\r-\k);
			}
		}

		\draw[black,line width=.85pt,rounded corners=5pt]
			({-sqrt(3)*\m},-\m)
			-- (2,{2*\h+2*\m})
			-- ({4+sqrt(3)*\m},-\m)
			-- cycle;

		\draw[black,double=white,double distance=1.4pt,line width=.5pt,rounded corners=5pt]
			({2-sqrt(3)*\m},{2*\h+\m})
			-- ({6+sqrt(3)*\m},{2*\h+\m})
			-- (4,{-2*\m})
			-- cycle;

		\foreach \r in {0,1,2}{
			\foreach \j in {-1,7}{
				\node[vertex,draw=black!25] at (v\r-\j) {};
			}
			\foreach \j in {0,...,6}{
				\node[vertex] at (v\r-\j) {};
			}
		}

		\foreach \j in {0,1,2}{
			\node[simple] at (v0-\j) {};
		}
	\end{scope}

	\foreach \u in {-1,13}{
		\draw[cut]
			({\u-.3},{-.3*\h})
			-- ({\u+2.3},{2.3*\h});
	}

	\node[below] at (0,-.95) {$S_1$};
	\node[below] at (2,-.95) {$S_2$};
	\node[below] at (4,-.95) {$S_3$};

	\node[below] at (2,-1.65) {$\mathcal H$};
	\node[above] at (5,{2*\h+1}) {$\Sigma^{-3}\mathcal H$};

\end{tikzpicture}
\caption{The AR quiver of $\C_d(kQ)$ for $Q=(1\to2\to3)$ and $d=-3$. The blue vertices are the standard simple-minded system $S_1,S_2,S_3$; the solid and double boundaries indicate the hearts $\mathcal H$ and $\Sigma^{-3}\mathcal H$, respectively.}
\label{fig:A3-negative-cluster}
\end{figure}

\subsection{Positive dg algebras}

A dg algebra $A$ is called \defterm{strictly positive} if $A^i=0$ for $i<0$, $d(A^0)=0$, and $A^0$ is semisimple. We shall only use strictly positive models with $A^0\simeq k^r$. Put $S_A=A^0$, viewed as a dg $A$-module through the canonical augmentation $A\to A^0$, and put $\Hcal_A=\Filt(\add A)$.

For a locally finite non-positive dg algebra $B$, let $S_B$ be the direct sum of the simple $H^0(B)$-modules, viewed in the standard heart of $\D(B)$.

For a locally finite positive dg algebra, \cite[Proposition~5.6]{KN13} gives $\pvd(A)=\thick(S_A)$, while \cite[Lemma~6.2]{KN13} gives $\Hom_{\D(A)}(S_A,\Sigma^iS_A)=0$ for $i>0$. Thus $S_A$ is a silting object of $\pvd(A)$. Consequently, whenever $S_A\in\per(A)$,
\begin{equation}\label{eq:pvd-thick}
\pvd(A)=\thick(S_A)\subseteq\per(A).
\end{equation}

\begin{thm}[{\cite[Theorem~D(2)]{Fus24}}]\label{thm:Koszul-dual}
Let $A$ be a pvd-finite locally finite strictly positive dg algebra with $A^0\simeq k^r$, and let
\[
B:=A^!:=\REnd_A(S_A)
\]
be its Koszul dual. Then $B$ is a locally finite non-positive dg algebra, and the Koszul dual functor $\Phi_A=\RHom_A(-,S_A)$ induces equivalences
\[
\pvd(A)\xrightarrow{\sim}\per(B^{\mathrm{o}})^{\mathrm{o}},
\qquad
\per(A)\xrightarrow{\sim}\pvd(B^{\mathrm{o}})^{\mathrm{o}},
\]
which send $S_A$ to $B^{\mathrm{o}}$ and $A$ to $S_{B^{\mathrm{o}}}$, respectively.
\end{thm}
Since $\Phi_A$ is contravariant and exact,
\[
\Phi_A(\Hcal_A)=\Filt(\add S_{B^{\mathrm{o}}})^{\mathrm{o}}
=(\modu H^0(B^{\mathrm{o}}))^{\mathrm{o}}.
\]
If the Koszul dual $B=A^!$ is proper, then $A$ is pvd-finite: indeed, $\pvd(A)=\thick(S_A)$ and properness of $B=\REnd_A(S_A)$ propagates by finite d\'evissage to Hom-finiteness of $\pvd(A)$. We shall also use that the equivalences of \cite[Theorem~D]{Fus24} lift to quasi-equivalences between the corresponding standard dg enhancements.

\begin{rmk}[Opposite conventions]\label{rmk:opposite-conventions}
With our right-module convention, there is a canonical quasi-isomorphism
\[
(A^{\mathrm{o}})^!\simeq (A^!)^{\mathrm{o}},
\]
cf. \cite[Remark~4.2]{Fus24}.
Thus, if $B=A^!$, the Koszul functor for $A^{\mathrm{o}}$ takes $S_{A^{\mathrm{o}}}$ to $((A^{\mathrm{o}})^!)^{\mathrm{o}}\simeq B$, not to a second copy of $B^{\mathrm{o}}$. Passing to opposites also preserves Jin's self-injectivity relation: $\add\DD B=\add\Sigma^{d+1}B$ if and only if $\add\DD(B^{\mathrm{o}})=\add\Sigma^{d+1}B^{\mathrm{o}}$.
\end{rmk}

\begin{rmk}\label{rmk:pvd-finite}
Suppose that $A$ is locally finite and $S_A\in\per(A)$. Then \eqref{eq:pvd-thick} gives $\pvd(A)\subseteq\per(A)$. Since local finiteness of $A$ implies that $\per(A)$ is Hom-finite, $\pvd(A)$ is Hom-finite as well; thus $A$ is pvd-finite and \cref{thm:Koszul-dual} applies. In this situation the two equivalences in \cref{thm:Koszul-dual} identify the inclusion $\pvd(A)\subseteq\per(A)$ with $\per(B^{\mathrm{o}})\subseteq\pvd(B^{\mathrm{o}})$ and induce an equivalence
\[
\cosg(A)\xrightarrow{\sim}\sg(B^{\mathrm{o}})^{\mathrm{o}}.
\]
\end{rmk}

\section{The cosingular fundamental domain}\label{sec:cosingular-fundamental-domain}

Throughout, let $A$ be a locally finite strictly positive dg algebra with $A^0\simeq k^r$ and put $B=A^!$.

\begin{dfn}\label{dfn:d-stable}
The dg algebra $A$ is called \defterm{$d$-stable} if
\begin{itemize}
\item [(i)] $S_A\in\Hcal_A\ast\Sigma^{-1}\Hcal_A\ast\cdots\ast\Sigma^{d+1}\Hcal_A$,
\item [(ii)] $H^i\RHom_A(S_A,A)=0$ if $i\neq d+1$.
\end{itemize}
\end{dfn}

Following \cite[Definition~2.2]{Jin20}, a non-positive dg algebra $B$ is called \defterm{$q$-self-injective}, for $q\geq1$, if
\[
\add B=\add\bigl(\Sigma^{q-1}\DD B\bigr)
\]
in $\D(B)$, equivalently $\add\DD B=\add(\Sigma^{1-q}B)$. We use this notion with $q=-d$.

\begin{prop}\label{prop:d-stable-selfinjective}
The dg algebra $A$ is $d$-stable if and only if its Koszul dual $B$ is $(-d)$-self-injective. In this case $B$ is automatically proper.
\end{prop}
\begin{proof}
Suppose first that $A$ is $d$-stable. By \cref{dfn:d-stable}(i), $S_A\in\per(A)$, so \cref{rmk:pvd-finite} applies. The duality
\[
\RHom_A(-,A):\per(A)^{\mathrm{o}}\xrightarrow{\sim}\per(A^{\mathrm{o}})
\]
sends the $r$ pairwise non-isomorphic indecomposable summands of $S_A$ to $r$ pairwise non-isomorphic indecomposable objects. By \cref{dfn:d-stable}(ii), these objects have cohomology concentrated in degree $d+1$. Since $H^0(A^{\mathrm{o}})=(A^0)^{\mathrm{o}}\simeq k^r$ is semisimple, they are, up to permutation, the corresponding shifts of the simple $A^{\mathrm{o}}$-modules. Hence
\begin{equation}\label{eq:dual-simple}
\add\RHom_A(S_A,A)=\add\bigl(\Sigma^{-d-1}S_{A^{\mathrm{o}}}\bigr).
\end{equation}
Applying $\Phi_{A^{\mathrm{o}}}=\RHom_{A^{\mathrm{o}}}(-,S_{A^{\mathrm{o}}})$, \cref{rmk:opposite-conventions} gives $\Phi_{A^{\mathrm{o}}}(S_{A^{\mathrm{o}}})\simeq B$, while tensor--Hom adjunction and the symmetric Frobenius identification $\DD S_A\simeq S_{A^{\mathrm{o}}}$ give
\[
\begin{aligned}
\Phi_{A^{\mathrm{o}}}\bigl(\RHom_A(S_A,A)\bigr)
&=\RHom_{A^{\mathrm{o}}}\bigl(\RHom_A(S_A,A),S_{A^{\mathrm{o}}}\bigr)\\
&\simeq\RHom_{A^{\mathrm{o}}}\bigl(\RHom_A(S_A,A),\DD S_A\bigr)\\
&\simeq\DD\bigl(S_A\otimes_A^{\mathbf L}\RHom_A(S_A,A)\bigr)\\
&\simeq\DD\RHom_A(S_A,S_A)=\DD B.
\end{aligned}
\]
For $M\in\per(A)$ and $N\in\D(A)$, the canonical tensor--evaluation morphism
\[
N\otimes_A^{\mathbf L}\RHom_A(M,A)\longrightarrow\RHom_A(M,N)
\]
is an isomorphism: this is clear for $M=A$, and the full subcategory of $M$ for which it is an isomorphism is thick. Applying this with $M=N=S_A$ gives the last isomorphism above.
Since \eqref{eq:dual-simple} also gives $S_{A^{\mathrm{o}}}\in\per(A^{\mathrm{o}})$, \cref{rmk:pvd-finite} applies to $A^{\mathrm{o}}$. Koszul duality reverses shifts, so \eqref{eq:dual-simple} gives
\[
\add\DD B=\add\bigl(\Sigma^{d+1}B\bigr).
\]
Thus $B$ is $(-d)$-self-injective. Since $B$ is non-positive and locally finite while $\DD B$ is non-negative, the equality $\add\DD B=\add(\Sigma^{d+1}B)$ forces $H^i(B)=0$ for $i\notin[d+1,0]$. Hence $B$ is proper.

Conversely, assume that $B$ is $(-d)$-self-injective. Thus
\[
\add\DD B=\add\bigl(\Sigma^{d+1}B\bigr).
\]
Since $B$ is non-positive and locally finite while $\DD B$ is non-negative, this equality forces $H^i(B)=0$ for $i\notin[d+1,0]$; in particular, $B$ is proper. Hence $A$ is pvd-finite by the observation after \cref{thm:Koszul-dual}, and \cref{thm:Koszul-dual} applies. Let $\Hcal_{B^{\mathrm{o}}}\subseteq\D(B^{\mathrm{o}})$ be the standard heart, identified with $\modu H^0(B^{\mathrm{o}})$. The standard truncation tower of $B^{\mathrm{o}}\in\pvd(B^{\mathrm{o}})^{[d+1,0]}$ gives
\[
B^{\mathrm{o}}\in
\Sigma^{-d-1}\Hcal_{B^{\mathrm{o}}}*\Sigma^{-d-2}\Hcal_{B^{\mathrm{o}}}*\cdots*\Sigma\Hcal_{B^{\mathrm{o}}}*\Hcal_{B^{\mathrm{o}}}.
\]
Since $\Phi_A$ is contravariant, sends $S_A$ to $B^{\mathrm{o}}$, and sends $\Hcal_A$ onto $\Hcal_{B^{\mathrm{o}}}^{\mathrm{o}}$, this yields
\[
S_A\in\Hcal_A*\Sigma^{-1}\Hcal_A*\cdots*\Sigma^{d+1}\Hcal_A.
\]
Thus \cref{dfn:d-stable}(i) holds. Finally, $(A^{\mathrm{o}})^!\simeq B^{\mathrm{o}}$ is proper, so the Koszul duality for $A^{\mathrm{o}}$ is available. Applying it to the self-injectivity equality recovers \eqref{eq:dual-simple}, and hence
\[
H^i\RHom_A(S_A,A)=0\qquad(i\neq d+1).
\]
Thus \cref{dfn:d-stable}(ii) holds.
\end{proof}

\begin{prop}[Proper Koszul duals]\label{prop:smooth-proper}
Let $A$ be a locally finite strictly positive dg algebra with $A^0\simeq k^r$, and assume that its Koszul dual $B=A^!$ is proper. Then $S_A\in\per(A)$, hence $\pvd(A)=\thick(S_A)\subseteq\per(A)$. Moreover, $A$ is homologically smooth and the dg category $\pvd_{\dg}(A)$ is proper. In particular, these conclusions hold for every $d$-stable $A$.
\end{prop}
\begin{proof}
Put $\ell=A^0\simeq k^r$. Since $B$ is proper, $A$ is pvd-finite by the observation after \cref{thm:Koszul-dual}. Under \cref{thm:Koszul-dual}, $S_A$ corresponds to the regular $B^{\mathrm{o}}$-module, which lies in $\pvd(B^{\mathrm{o}})$; hence $S_A\in\per(A)$. Thus \eqref{eq:pvd-thick} gives
\[
\pvd(A)=\thick(S_A)\subseteq\per(A).
\]

The enveloping dg algebra
\[
A^e=A^{\mathrm{o}}\otimes_k A
\]
is again locally finite strictly positive with $(A^e)^0\simeq\ell^e$. By the K\"unneth formula, its Koszul dual is quasi-isomorphic, up to the canonical opposite convention, to
\[
(A^e)^!\simeq B^{\mathrm{o}}\otimes_k B=B^e.
\]
Hence $(A^e)^!$ is proper. By the same observation, $A^e$ is pvd-finite, so the full Koszul duality of \cite[Theorem~D(2)]{Fus24} applies to $A^e$.

For this proof, write
\[
\D^+_{\mathrm{fd}}(C)=\{X\in\D(C)\mid H^i(X)=0\text{ for }i\ll0,\ \dim_kH^i(X)<\infty\text{ for all }i\},
\]
with $\D^-_{\mathrm{fd}}(C)$ defined dually. The diagonal bimodule $A$, regarded as a right $A^e$-module, belongs to $\D^+_{\mathrm{fd}}(A^e)$. Via the standard symmetric Frobenius form on $\ell\simeq k^r$, we have $\DD\ell\simeq \ell$ as $\ell$-bimodules, and hence
\[
\RHom_k(\ell,\ell)\simeq \DD\ell\otimes_k\ell\simeq \ell^e=S_{A^e}
\]
as right $A^e$-modules. Tensor--Hom adjunction therefore gives
\[
\Phi_e(A):=\RHom_{A^e}(A,S_{A^e})
\simeq
\RHom_{A^e}\bigl(A,\RHom_k(\ell,\ell)\bigr)
\simeq
\RHom_A(\ell,\ell)
=
B.
\]
In particular, $\Phi_e(A)$ has finite-dimensional total cohomology because $B$ is proper. By \cite[Theorem~D(2)]{Fus24}, the Koszul duality equivalence
\[
\D^+_{\mathrm{fd}}(A^e)\xrightarrow{\sim}
\D^-_{\mathrm{fd}}\bigl(((A^e)^!)^{\mathrm{o}}\bigr)^{\mathrm{o}}
\]
identifies $\per(A^e)$ with the perfectly valued subcategory on the Koszul-dual side. Since $\Phi_e(A)$ belongs to that subcategory, we obtain
\[
A\in\per(A^e).
\]
Thus $A$ is homologically smooth.

Finally, Koszul duality gives a dg quasi-equivalence
\[
\pvd_{\dg}(A)\simeq \per_{\dg}(B^{\mathrm{o}})^{\mathrm{o}}.
\]
Since $B$ is proper, the dg category $\per_{\dg}(B^{\mathrm{o}})$ is proper. Hence so is $\pvd_{\dg}(A)$.
\end{proof}

If $A$ is $d$-stable, then $B^{\mathrm{o}}$ is proper and $(-d)$-self-injective by \cref{prop:d-stable-selfinjective} and \cref{rmk:opposite-conventions}; in particular it satisfies Assumption~0.1 of \cite{Jin20}, and \cite[Proposition~2.3]{Jin20} gives
\[
\CM(B^{\mathrm{o}})=\pvd(B^{\mathrm{o}})^{[d+1,0]}.
\]
We record the resulting Cohen--Macaulay description, which will be used repeatedly. By \cref{rmk:pvd-finite} and \cite[Theorem~2.4]{Jin20}, the first equivalence is
\begin{equation}\label{eq:CM-opposite-identification}
\cosg(A)\xrightarrow{\sim}\underline{\CM}(B^{\mathrm{o}})^{\mathrm{o}}
\xrightarrow{\sim}\underline{\CM}(B).
\end{equation}
For the second equivalence, \cite[Proposition~2.6]{Jin20} gives an exact duality
\[
\RHom_{B^{\mathrm{o}}}(-,B^{\mathrm{o}}):
\CM(B^{\mathrm{o}})^{\mathrm{o}}\xrightarrow{\sim}\CM(B),
\]
which sends the projective-injective subcategory $\add B^{\mathrm{o}}$ to $\add B$. Hence it descends to the stable categories and gives the second equivalence in \eqref{eq:CM-opposite-identification}.

Put
\[
\F_A:=\F^{d+1}(\Hcal_A).
\]

\begin{prop}\label{prop:fundamental-domain}
Assume that $A$ is $d$-stable.
Then the following hold.
\begin{itemize}
\item [(1)] The quotient functor $\per(A)\to\cosg(A)$ induces an equivalence
\[
\F_A/[\add S_A]\xrightarrow{\sim}\cosg(A).
\]
\item [(2)] For every $X,Y\in\F_A$ and every $i>0$, the map
\[
\Hom_{\D(A)}(X,\Sigma^iY)\longrightarrow\Hom_{\cosg(A)}(X,\Sigma^iY)
\]
is an isomorphism.
\end{itemize}
\end{prop}
\begin{proof}
By \cref{prop:d-stable-selfinjective} and the Cohen--Macaulay description above, \cref{thm:Koszul-dual} gives
\[
\Phi_A(\F_A)=\CM(B^{\mathrm{o}})^{\mathrm{o}}.
\]
Now (1) follows from \cite[Theorem~2.4]{Jin20}, and (2) follows from \cite[Proposition~2.6]{Jin20}.
\end{proof}

\section{Cosingular realization}\label{sec:cosingular-realization}
\begin{thm}[Cosingular realization]\label{thm:realization}
Let $d\leq-2$, and let $\C$ be a Hom-finite algebraic triangulated category equipped with a pretriangulated dg enhancement $\C_{\dg}$. Let $\Lcal=\{L_1,\dots,L_r\}$ be a $(-d)$-simple-minded system in $\C$, and put $L=\bigoplus_{i=1}^rL_i$. Then there exist a $d$-stable locally finite strictly positive dg algebra $A$ with $A^0\simeq k^r$ and an exact quasi-functor $\pi_{\dg}:\per_{\dg}(A)\to\C_{\dg}$ which induces a quasi-equivalence
\[
\overline\pi_{\dg}:\cosg_{\dg}(A)\simeq\C_{\dg}.
\]
\end{thm}

\begin{lem}\label{lem:construct-A}
There exist a locally finite strictly positive dg algebra $A$ and an exact quasi-functor $\pi_{\dg}:\per_{\dg}(A)\to\C_{\dg}$ such that $H^0(\pi_{\dg})(A)\simeq L$ and
\[
H^i(A)\xrightarrow{\sim}\Hom_{\C}(L,\Sigma^iL)\qquad(i\geq d+1).
\]
\end{lem}

\begin{proof}
After replacing $\C_{\dg}$ by a quasi-equivalent pretriangulated dg enhancement, choose $L=\bigoplus_iL_i$ as a strict finite direct sum. Put $R=\REnd_{\C_{\dg}}(L)$. Then $H^i(R)\simeq\Hom_{\C}(L,\Sigma^iL)$ and $H^0(R)\simeq k^r$, while the summand identities give closed orthogonal idempotents $e_1,\dots,e_r\in Z^0(R)$. Set $\ell=\bigoplus_i ke_i\simeq k^r$. The image $d(R^0)\subseteq R^1$ is an $\ell$-$\ell$-subbimodule. Since $\ell^e$ is semisimple, choose an $\ell$-$\ell$-subbimodule $V\subseteq R^1$ such that $R^1=d(R^0)\oplus V$.

Define a dg subalgebra $A\subseteq R$ by $A^i=0$ for $i<0$, $A^0=\ell$, $A^1=V$, and $A^i=R^i$ for $i\geq2$. The bimodule condition on $V$ makes $A$ a dg subalgebra. The inclusion $A\hookrightarrow R$ induces an isomorphism on cohomology in every nonnegative degree. For $d+1\leq i<0$, the left-hand side vanishes by construction, while the right-hand side vanishes by \cref{dfn:sms}(ii). This gives the displayed comparison. Since $\C$ is Hom-finite, $A$ is locally finite.
The morphism $A\to R$ gives a quasi-functor $\per_{\dg}(A)\to\per_{\dg}(R)$. The canonical dg Yoneda functor $\per_{\dg}(R)\to\C_{\dg}$ has essential image the dg subcategory split-generated by $L$, whose homotopy category is $\thick_{\C}(L)=\C$ by \cref{dfn:sms}(iii). Hence it is a quasi-equivalence, and composition gives the required quasi-functor $\pi_{\dg}$.
\end{proof}

\begin{lem}\label{lem:hom-comparison}
Let $\pi=H^0(\pi_{\dg})\colon\per(A)\to\C$. The following hold.
\begin{enumerate}[label=\textup{(\roman*)}]
	\item For $P,Q\in\Hcal_A$ and $i\geq d+1$, the canonical map
	\[
	\Hom_{\D(A)}(P,\Sigma^iQ)\longrightarrow\Hom_{\C}(\pi P,\Sigma^i\pi Q)
	\]
	is an isomorphism.
	\item For $X\in\Sigma^{-d-1}\F_A$, $Q\in\Hcal_A$, and $i>0$, the canonical map
	\[
	\Hom_{\D(A)}(X,\Sigma^iQ)\longrightarrow\Hom_{\C}(\pi X,\Sigma^i\pi Q)
	\]
	is an isomorphism.
	\item For $P\in\Hcal_A$, $Y\in\F_A$, and $i>0$, the canonical map
	\[
	\Hom_{\D(A)}(P,\Sigma^iY)\longrightarrow\Hom_{\C}(\pi P,\Sigma^i\pi Y)
	\]
	is an isomorphism.
\end{enumerate}
\end{lem}
\begin{proof}
For $P,Q\in\add A$, part (i) is \cref{lem:construct-A}. For $d+1\leq i<0$, both sides vanish for arbitrary $P,Q\in\Hcal_A$: on the left by a finite d\'evissage from the positivity of $A$, and on the right by \cref{dfn:sms}(ii). For $i\geq0$, a double induction on the extension lengths of $P$ and $Q$, using the long exact Hom sequences, proves (i). At $i=0$, the adjacent degree $-1$ term is covered by the preceding vanishing.

For (ii), we induct on the number of extension layers of $X$ in
\[
\Sigma^{-d-1}\F_A
=
\Sigma^{-d-1}\Hcal_A*\Sigma^{-d-2}\Hcal_A*\cdots*\Hcal_A.
\]
For a single layer, write $X=\Sigma^aP$ with $P\in\Hcal_A$ and $0\leq a\leq-d-1$. Then the required comparison in degree $i>0$ is part (i) in degree $i-a\geq d+2$. For the induction step, peel off the leftmost layer and write $X$ in a triangle
\[
\Sigma^aP\longrightarrow X\longrightarrow X'\longrightarrow\Sigma^{a+1}P,
\]
where $0\leq a\leq-d-1$ and $X'$ has fewer layers. Applying $\Hom(-,\Sigma^iQ)$ gives the exact sequence
\[
\begin{aligned}
\Hom(\Sigma^{a+1}P,\Sigma^iQ)&\longrightarrow\Hom(X',\Sigma^iQ)\longrightarrow\Hom(X,\Sigma^iQ)\\
&\longrightarrow\Hom(\Sigma^aP,\Sigma^iQ)\longrightarrow\Hom(X',\Sigma^{i+1}Q),
\end{aligned}
\]
and the same sequence after applying $\pi$. The comparison maps for the peeled factor are isomorphisms by part (i) in degrees $i-a-1$ and $i-a$, both at least $d+1$, while the two maps involving $X'$ are isomorphisms by the induction hypothesis. Hence the middle comparison is an isomorphism.

Part (iii) follows by the same argument in the second variable. Namely, induct on the extension layers of $Y\in\F_A$ and peel off a factor $\Sigma^aQ$ from the quotient end, with $d+1\leq a\leq0$. The two adjacent comparisons for the peeled factor occur in degrees $i+a-1$ and $i+a$, which are again at least $d+1$ for $i>0$; part (i) and the induction hypothesis therefore give the result.
\end{proof}

\begin{lem}\label{lem:density}
For every $a\in\mathbb Z$ and every $d+1\leq j\leq0$, the functor
\[
\pi:\Sigma^a\F^j(\Hcal_A)\longrightarrow\Sigma^a\F^j(\Hcal_{\Lcal})
\]
is dense.
\end{lem}

\begin{proof}
    By shifting, it is enough to treat $a=0$. First, $\pi:\Hcal_A\to\Hcal_{\Lcal}$ is dense. Indeed, each simple $L_i$ is the image of $e_iA$, and an induction on the length in $\Hcal_{\Lcal}$ lifts extensions by \cref{lem:hom-comparison}(i) with $i=1$.

    We now induct on the number of factors in the extension window. Let $X\in\F^j(\Hcal_{\Lcal})$. By definition, $X$ occurs in a triangle
    \[
    U_0\longrightarrow X\longrightarrow V_0\longrightarrow\Sigma U_0
    \]
    with $U_0\in\F^{j+1}(\Hcal_{\Lcal})$ and $V_0\in\Sigma^j\Hcal_{\Lcal}$. By the induction hypothesis, choose lifts $U\in\F^{j+1}(\Hcal_A)$ and $V\in\Sigma^j\Hcal_A$ with $\pi(U)\simeq U_0$ and $\pi(V)\simeq V_0$. After these identifications, write the triangle as
    \[
    \pi(U)\longrightarrow X\longrightarrow \pi(V)\xrightarrow{\xi}\Sigma\pi(U).
    \]
    Writing $V=\Sigma^jP$ with $P\in\Hcal_A$, the connecting morphism corresponds to an element of
    \[
    \Hom_{\C}(\pi P,\Sigma^{1-j}\pi U).
    \]
    Since $1-j>0$ and $U\in\F_A$, \cref{lem:hom-comparison}(iii) lifts this element to a morphism $V\to\Sigma U$. Complete the lifted morphism to a triangle
    \[
    U\longrightarrow W\longrightarrow V\longrightarrow\Sigma U.
    \]
    Then $W\in\F^j(\Hcal_A)$ and $\pi(W)\simeq X$.
\end{proof}

\begin{lem}\label{lem:construct-SA}
The canonical semisimple module satisfies
\[
S_A\in\F_A,\qquad \pi(S_A)=0.
\]
\end{lem}

\begin{proof}
    Let $L\mono I$ be the minimal injective envelope from \cref{lem:injectives}. By \cref{lem:density}, choose $\widetilde I\in\Hcal_A$ with $\pi(\widetilde I)\simeq I$. Since
    \[
    I\in\Sigma^{-2}\F^{d+2}(\Hcal_{\Lcal}),
    \]
    the shifted form of \cref{lem:density} gives
    \[
    \widetilde X\in\Sigma^{-2}\F^{d+2}(\Hcal_A),
    \qquad \pi(\widetilde X)\simeq I.
    \]
    Since
    \[
    \Sigma^{-d}\widetilde X
    \in\Sigma^{-d-2}\F^{d+2}(\Hcal_A)
    \subseteq\Sigma^{-d-1}\F_A,
    \]
    \cref{lem:hom-comparison}(ii), applied with $i=-d>0$, gives directly
    \[
    \Hom_{\D(A)}(\widetilde X,\widetilde I)\xrightarrow{\sim}\Hom_{\C}(\pi\widetilde X,\pi\widetilde I).
    \]
    Thus the isomorphism $\pi(\widetilde X)\simeq\pi(\widetilde I)$ lifts to a morphism $f:\widetilde X\to\widetilde I$. Complete it to a triangle
    \[
    \widetilde X\xto{f}\widetilde I\longrightarrow S\longrightarrow\Sigma\widetilde X.
    \]
    Then $\pi(S)=0$ and $S\in\F_A$, so $H^{<0}(S)=0$.
    
    By \cref{lem:hom-comparison}(iii), $\Hom_{\D(A)}(A,\Sigma^iS)=0$ for every $i>0$, hence $H^{>0}(S)=0$. Thus $S$ has cohomology only in degree zero. Since $\widetilde X\in\Sigma^{-2}\F^{d+2}(\Hcal_A)$, positivity gives $\Hom_{\D(A)}(A,\widetilde X)=0=\Hom_{\D(A)}(A,\Sigma\widetilde X)$. Applying $\Hom_{\D(A)}(A,-)$ to the defining triangle and using \cref{lem:hom-comparison}(i) yields
    \[
    H^0(S)\simeq\Hom_{\D(A)}(A,\widetilde I)\simeq\Hom_{\C}(L,I)\simeq k^r
    \]
    as right $H^0(A)=k^r$-modules. The uniqueness of simple dg modules over positive dg algebras \cite[Corollary~5.7]{KN13}, applied componentwise, gives $S\simeq S_A$.
    \end{proof}

\begin{lem}\label{lem:dual-concentration}
    The canonical semisimple module satisfies
    \[
    H^i\RHom_A(S_A,A)=0\qquad(i\neq d+1).
    \]
\end{lem}
\begin{proof}
    By \cref{lem:construct-SA}, the object $S_A$ lies in the fundamental extension window. Applying $\RHom_A(-,A)$ gives
    \[
    \RHom_A(S_A,A)\in\Sigma^{-d-1}\Hcal_{A^{\mathrm{o}}}*\Sigma^{-d-2}\Hcal_{A^{\mathrm{o}}}*\cdots*\Hcal_{A^{\mathrm{o}}}\subseteq\D(A^{\mathrm{o}})^{\geq d+1}.
    \]
    Therefore $H^i\RHom_A(S_A,A)=0$ for $i<d+1$.
    
    For the other direction, $\Sigma^{-d-1}S_A\in\Sigma^{-d-1}\Hcal_A*\Sigma^{-d-2}\Hcal_A*\cdots*\Hcal_A$. By \cref{lem:hom-comparison}(ii), for every $j>0$ one has
    \[
    \Hom_{\D(A)}(\Sigma^{-d-1}S_A,\Sigma^jA)\simeq\Hom_{\C}(\Sigma^{-d-1}\pi S_A,\Sigma^jL)=0.
    \]
    This is precisely the vanishing of $H^i\RHom_A(S_A,A)$ for $i>d+1$.
\end{proof}
\begin{proof}[Proof of \cref{thm:realization}]
    The dg algebra $A$ and quasi-functor $\pi_{\dg}$ are given by \cref{lem:construct-A}. By \cref{lem:construct-SA,lem:dual-concentration}, $A$ is $d$-stable. By \eqref{eq:pvd-thick}, one has $\pvd(A)=\thick(S_A)\subseteq\per(A)$, and $\pi_{\dg}$ factors through a quasi-functor $\overline\pi_{\dg}:\cosg_{\dg}(A)\to\C_{\dg}$.
    
    By \cref{prop:fundamental-domain}, every object of $\cosg(A)$ is represented by an object of $\F_A$. Combining the morphism comparison in \cref{prop:fundamental-domain} with \cref{lem:hom-comparison}(iii) gives, for every $Y\in\F_A$, an isomorphism
    \[
    \Hom_{\cosg(A)}(A,\Sigma Y)\xrightarrow{\sim}\Hom_{\C}(L,\Sigma\overline\pi Y).
    \]
    Given $Y\in\F_A$, choose $Y'\in\F_A$ representing $\Sigma^{-1}Y$ in $\cosg(A)$. Applying the displayed degree-one comparison to $Y'$ yields
    \[
    \Hom_{\cosg(A)}(A,Y)\xrightarrow{\sim}\Hom_{\C}(L,\overline\pi Y).
    \]
    Let $\mathcal U$ be the full subcategory of objects $X\in\per(A)$ such that
    \[
    \Hom_{\cosg(A)}(X,Z)\longrightarrow\Hom_{\C}(\overline\pi X,\overline\pi Z)
    \]
    is an isomorphism for every $Z\in\cosg(A)$. Long exact Hom sequences and the five lemma show that $\mathcal U$ is triangulated, and it is closed under direct summands; hence $\mathcal U$ is thick. The preceding degree-zero comparison, together with the fundamental-domain density, gives $A\in\mathcal U$. Therefore $\mathcal U=\thick(A)=\per(A)$, so $H^0(\overline\pi_{\dg})$ is fully faithful. Density follows from \cref{lem:density} with $j=d+1$ and \cref{dfn:sms}(iii), so it is a triangle equivalence. Since both dg categories are pretriangulated, $\overline\pi_{\dg}$ is therefore a quasi-equivalence.
\end{proof}

\begin{cor}[Characterization by self-injective dg algebras]\label{cor:CM-characterization}
Let $d\leq-2$, and let $\C$ be a Hom-finite algebraic triangulated category. Then $\C$ admits a $(-d)$-simple-minded system if and only if it is triangle equivalent to $\underline{\CM}(B)$ for some locally finite $(-d)$-self-injective non-positive dg algebra $B$.
\end{cor}
\begin{proof}
For the forward implication, let $A$ be the dg algebra given by \cref{thm:realization} and put $B=A^!$. Then \eqref{eq:CM-opposite-identification} gives
\[
\C\simeq\cosg(A)\simeq\underline{\CM}(B).
\]
Conversely, \cite[Theorem~5.6]{Jin20} shows that the simple $H^0(B)$-modules give a $(-d)$-simple-minded system in $\underline{\CM}(B)$.
\end{proof}

\section{Acyclic negative Calabi--Yau categories}\label{sec:acyclic-negative-CY}

Let $d\leq-2$ and let $Q$ be a finite acyclic quiver. Let $S_Q$ be the direct sum of the simple right $kQ$-modules, and set
\[
kQ^!:=\Ext^*_{kQ}(S_Q,S_Q)
\]
with Yoneda multiplication. Since $kQ$ is hereditary, $kQ^!$ is concentrated in degrees $0$ and $1$, and its degree-one extension quiver is $Q$. The orbit category $\C_d(kQ)=\D^b(kQ)/(\Sigma^{-d}\Serre)$ is triangulated and Hom-finite in this hereditary setting; see \cite{Keller05,CoelhoPauksztelloPloog22}. The following is a negative analogue of the recognition theorem of Keller--Reiten \cite{KR08}.

\begin{thm}[Negative Keller--Reiten recognition theorem]\label{thm:negative-KR}
Let $d\leq-3$, and let $\C$ be a Hom-finite algebraic $d$-Calabi--Yau triangulated category with a $(-d)$-simple-minded system $\Lcal=\{L_1,\dots,L_r\}$. Put $L=\bigoplus_iL_i$, and let $Q$ be a finite acyclic quiver. Assume that
\[
\bigoplus_{i=0}^{-d-1}\Hom_{\C}(L,\Sigma^iL)\simeq kQ^!
\]
as graded algebras. Then
\[
\C\simeq\C_d(kQ).
\]
\end{thm}

\begin{rmk}[The hypothesis]\label{rmk:KR-hypothesis}
Because $kQ^!$ is concentrated in degrees $0$ and $1$, the hypothesis of \cref{thm:negative-KR} is equivalently the conjunction of
\[
\Hom_\C(L,\Sigma^iL)=0\qquad(2\le i\le-d-1)
\]
and the statement that the degree-one extension quiver of $\Hcal_{\Lcal}$ is $Q$. There is no further multiplicative condition: products of two degree-one classes land in degree two and hence vanish.
\end{rmk}

\begin{rmk}[The standard system]\label{rmk:KR-standard-system}
The standard simple-minded system of $\C_d(kQ)$ satisfies the hypothesis of \cref{thm:negative-KR}. Put $q=-d$. The simple $kQ$-modules form a simple-minded collection in $\D^b(kQ)$ contained in the standard fundamental domain, and their images form a $q$-simple-minded system in $\C_d(kQ)$ by \cite[Theorem~4.1]{CoelhoPauksztelloPloog22}. To compute the relevant orbit-category morphisms, note that
\[
\Serre\bigl(\D^{[a,b]}(kQ)\bigr)\subseteq\D^{[a-1,b]}(kQ),
\qquad
\Serre^{-1}\bigl(\D^{[a,b]}(kQ)\bigr)\subseteq\D^{[a,b+1]}(kQ),
\]
because $kQ$ is hereditary. Hence, for $0\leq i\leq q-1$, every orbit summand indexed by $n\neq0$
\[
\Hom_{\D^b(kQ)}\bigl(L,\Sigma^{qn+i}\Serre^nL\bigr)
\]
vanishes by amplitude: for $n\geq1$ the target lies strictly in negative cohomological degrees, whereas for $n\leq-1$ it lies strictly in positive cohomological degrees. Thus only the ordinary derived-category summand contributes, and
\[
\bigoplus_{i=0}^{q-1}\Hom_{\C_d(kQ)}(L,\Sigma^iL)
\simeq\End_{kQ}(L)\oplus\Ext^1_{kQ}(L,L)\simeq kQ^!
\]
as graded algebras.
\end{rmk}

\begin{cor}[Recognition characterization]\label{cor:negative-KR-characterization}
Let $d\leq-3$ and let $\C$ be a Hom-finite algebraic $d$-Calabi--Yau triangulated category. Then $\C\simeq\C_d(kQ)$ if and only if $\C$ admits a $(-d)$-simple-minded system $\Lcal$, with $L=\bigoplus_{X\in\Lcal}X$, such that
\[
\bigoplus_{i=0}^{-d-1}\Hom_{\C}(L,\Sigma^iL)\simeq kQ^!
\]
as graded algebras.
\end{cor}
\begin{proof}
The forward implication follows from \cref{rmk:KR-standard-system}, and the reverse implication is \cref{thm:negative-KR}.
\end{proof}

For the proof of \cref{thm:negative-KR}, retain its assumptions and put $\Hcal=\Hcal_{\Lcal}$. We divide the proof into two steps. First we identify the heart determined by $\Lcal$.

\begin{lem}\label{lem:KR-heart}
    Under the assumptions of \cref{thm:negative-KR}, let $L\mono I$ be the minimal injective envelope in $\Hcal$ and put $P=\Sigma^{-d}I$. Then $P$ is a basic projective generator of $\Hcal$, the category $\Hcal$ is hereditary, and, after fixing the labelling by the $L_i$, there is an equivalence
    \[
        \Hcal\simeq\modu kQ
    \]
    which sends $P$ to the regular right $kQ$-module.
\end{lem}

\begin{proof}
    Combining \cref{dfn:sms}(ii) with the fact that $kQ^!$ is concentrated in degrees zero and one, finite d\'evissage in both variables gives the single vanishing range
    \[
        \Hom_{\C}(\Hcal,\Sigma^i\Hcal)=0
        \qquad
        (d+1\leq i\leq-d-1,\ i\neq0,1).
    \]
    By \cref{lem:injectives},
    \[
        I\in\Hcal\cap(\Sigma^{-2}\Hcal*\cdots*\Sigma^d\Hcal).
    \]
    Since $\Hom_{\C}(\Sigma^{-j}\Hcal,I)=0$ for $2\leq j\leq-d-1$, we may successively remove the factors $\Sigma^{-2}\Hcal,\ldots,\Sigma^{d+1}\Hcal$: in a triangle $U\to I\to V\to\Sigma U$ with $U$ in the layer being removed, the first map is zero, so $I$ is a direct summand of $V$; the remaining window is summand-closed by \cref{lem:sms-window-summands}, after shifting. Hence $I\in\Sigma^d\Hcal$.

    Put $P=\Sigma^{-d}I$. The $d$-Calabi--Yau duality gives
    \[
        \DD\Hom_{\C}(P,\Sigma X)
        \simeq
        \Hom_{\C}(X,\Sigma^{-1}I)=0
        \qquad
        (X\in\Hcal),
    \]
    so $P$ is projective. Writing $I=\bigoplus_iI_i$ and $P=\bigoplus_iP_i$, we have
    \[
        \DD\Hom_{\Hcal}(P_i,L_j)
        \simeq
        \Hom_{\Hcal}(L_j,I_i)
        \simeq
        \delta_{ij}k.
    \]
    Thus the $P_i$ are precisely the indecomposable projectives corresponding to the simples $L_i$. Hence $P$ is a basic projective generator of $\Hcal$.

    Since $d\leq-3$, the vanishing above gives $\Hom_{\C}(\Hcal,\Sigma^2\Hcal)=0$. For $X,Y\in\Hcal$, the exact structure inherited from $\C$ gives the standard identification
    \[
        \Ext^1_{\Hcal}(X,Y)\simeq\Hom_{\C}(X,\Sigma Y).
    \]
    If $0\to Y'\to Y\to Y''\to0$ is exact in $\Hcal$, the corresponding triangle yields
    \[
        \Ext^1_{\Hcal}(X,Y)\longrightarrow\Ext^1_{\Hcal}(X,Y'')\longrightarrow\Hom_{\C}(X,\Sigma^2Y')=0.
    \]
    Thus $\Ext^1_{\Hcal}(X,-)$ is right exact for every $X$, which is equivalent to $\Hcal$ being hereditary.

    Set $\Gamma=\End_{\Hcal}(P)$. Then $\Hcal\simeq\modu\Gamma$. The algebra $\Gamma$ is basic and hereditary, and the degree-one part of the assumed isomorphism with $kQ^!$ identifies its Ext quiver with $Q$. Hence $\Gamma\simeq kQ$. Under the resulting equivalence $\Hcal\simeq\modu kQ$, the object $P$ corresponds to the regular right module $kQ$.
\end{proof}

The second step identifies the Koszul dual of the realizing positive dg algebra. We first isolate the bimodule compatibility supplied by the Calabi--Yau structure.

\begin{lem}\label{lem:CY-bimodule}
Let $P\in\C$ and put $\Gamma=\End_\C(P)$. If $\C$ is $d$-Calabi--Yau, then the bifunctorial Calabi--Yau isomorphism induces a $\Gamma$-$\Gamma$-bimodule isomorphism
\[
\Hom_\C(P,\Sigma^dP)\xrightarrow{\sim}\DD\Gamma.
\]
Here $\Gamma$ acts on $\Hom_\C(P,\Sigma^dP)$ by
\[
a\cdot x\cdot b=(\Sigma^da)\circ x\circ b,
\]
and on $\DD\Gamma$ by $(a\cdot\varphi\cdot b)(c)=\varphi(bca)$.
\end{lem}

\begin{proof}
Let $\langle-,-\rangle_P:\Gamma\times\Hom_\C(P,\Sigma^dP)\to k$ be the perfect pairing corresponding to the $d$-Calabi--Yau isomorphism. Bifunctoriality in both variables gives, for $a,b,c\in\Gamma$ and $x:P\to\Sigma^dP$,
\[
\langle c,(\Sigma^da)\circ x\circ b\rangle_P
=\langle bca,x\rangle_P.
\]
Therefore the map $x\mapsto(c\mapsto\langle c,x\rangle_P)$ is compatible with both the left and right $\Gamma$-actions, and it is an isomorphism because the Calabi--Yau pairing is non-degenerate.
\end{proof}

\begin{lem}\label{lem:KR-koszul-model}
    Let $A$ be the strictly positive dg algebra given by \cref{thm:realization} and put $B=A^!$. Under the assumptions of \cref{thm:negative-KR}, the dg algebra $B$ is formal and
    \[
        B\simeq kQ\ltimes \Sigma^{-d-1}\DD(kQ).
    \]
\end{lem}

\begin{proof}
    By \cref{lem:density,lem:hom-comparison}, the restriction
    \[
        \pi:\Hcal_A\xrightarrow{\sim}\Hcal
    \]
    is an equivalence. Choose lifts $\widetilde P,\widetilde I\in\Hcal_A$ of the projective generator $P$ and the injective cogenerator $I$. Since $I\simeq\Sigma^dP$, \cref{lem:hom-comparison}(i) lifts this isomorphism to a morphism
    \[
        \Sigma^d\widetilde P\longrightarrow\widetilde I.
    \]
    The argument of \cref{lem:construct-SA} identifies its cone with $S_A$, and therefore we obtain a triangle
    \[
        \Sigma^d\widetilde P
        \longrightarrow
        \widetilde I
        \longrightarrow
        S_A
        \longrightarrow
        \Sigma^{d+1}\widetilde P.
    \]

    The heart anti-equivalence induced by $\Phi_A$ sends the projective generator $\widetilde P$ to the injective cogenerator and the injective cogenerator $\widetilde I$ to the projective generator on the Koszul-dual side. After the canonical opposite identifications,
    \[
        \Phi_A(\widetilde P)\simeq\DD H^0(B),
        \qquad
        \Phi_A(\widetilde I)\simeq H^0(B).
    \]
    Since $\pi(\widetilde P)\simeq P$ and $\End_{\Hcal}(P)\simeq kQ$ by \cref{lem:KR-heart}, full faithfulness on the heart gives
    \[
        H^0(B)\simeq\End_{\Hcal_A}(\widetilde P)\simeq kQ.
    \]
    Apply $\Phi_A$ to the rotated triangle
    \[
        \widetilde I\longrightarrow S_A\longrightarrow\Sigma^{d+1}\widetilde P\longrightarrow\Sigma\widetilde I.
    \]
    Using $\Phi_A(S_A)\simeq B^{\mathrm{o}}$ and the identifications above, and then passing through the canonical opposite convention, we obtain a triangle
    \begin{equation}\label{eq:KR-koszul-triangle}
        \Sigma^{-d-1}\DD(kQ)
        \longrightarrow
        B
        \longrightarrow
        kQ
        \longrightarrow
        \Sigma^{-d}\DD(kQ).
    \end{equation}
    Since $-d\geq3$, its long exact cohomology sequence gives
    \[
        H^i(B)=0\quad(i\neq0,d+1),
        \qquad
        H^0(B)\simeq kQ,
        \qquad
        H^{d+1}(B)\simeq\DD(kQ)
    \]
    as one-sided modules.

    It remains to identify the bimodule structure in degree $d+1$. Naturality of the Hom identification underlying $\Phi_A$ with respect to both endomorphism actions identifies
    \[
        H^{d+1}(B)\simeq\Hom_{\C}(P,I)=\Hom_{\C}(P,\Sigma^dP)
    \]
    as a $\Gamma$-$\Gamma$-bimodule, where $\Gamma=\End_{\Hcal}(P)\simeq kQ$, with action
    \[
        a\cdot x\cdot b=(\Sigma^da)\circ x\circ b.
    \]
    By \cref{lem:CY-bimodule},
    \[
        H^{d+1}(B)\simeq\Hom_{\C}(P,\Sigma^dP)\simeq\DD\Gamma\simeq\DD(kQ)
    \]
    as $kQ$-bimodules. Thus no Nakayama twist occurs. Since $H^{2d+2}(B)=0$, the product of two elements of $H^{d+1}(B)$ vanishes. Thus, as a graded algebra,
    \begin{equation}\label{eq:KR-cohomology-algebra}
        H^*(B)\simeq kQ\ltimes \Sigma^{-d-1}\DD(kQ).
    \end{equation}

    It remains to prove formality. Put
    \[
    \ell=k^{Q_0},
    \qquad
    \Gamma=kQ,
    \qquad
    M=\Sigma^{-d-1}\DD\Gamma,
    \qquad
    T=\Gamma\ltimes M,
    \]
    and regard $T$ as an augmented graded algebra over $\ell$ via $T\to\Gamma\to\ell$. Here $\Gamma=kQ$ is the tensor algebra over $\ell$ on $V=kQ_1$. We apply \cite[Theorem~4.7]{SeidelThomas01}, whose semisimple ground ring $R=k^{Q_0}$ is precisely our $\ell$. For $n>2$, degree considerations give $C^n_{\ell}(T,T[2-n])=0$ unless $n=1-d$: a nonzero degree-zero cochain must take only degree-zero inputs and have its image in $M$. For the exceptional shift $2-n=d+1$, the whole cochain complex is
    \[
    C^s_{\ell}\bigl(T,T[d+1]\bigr)=\Hom_{\ell^e}\bigl(\overline\Gamma^{\otimes_{\ell}s},M\bigr),
    \qquad \overline\Gamma=\ker(\Gamma\to\ell),
    \]
    with the relative Hochschild differential for $\Gamma$ and coefficients $M=\Sigma^{-d-1}\DD\Gamma$. Thus the relevant cohomology group is $\operatorname{HH}^{1-d}_{\ell}(\Gamma,\DD\Gamma)$. Since $\Gamma$ is the tensor algebra over $\ell$ on $V$, it has the relative bimodule resolution
    \[
        0\longrightarrow
        \Gamma\otimes_{\ell}V\otimes_{\ell}\Gamma
        \longrightarrow
        \Gamma\otimes_{\ell}\Gamma
        \longrightarrow
        \Gamma
        \longrightarrow0,
    \]
    one has $\operatorname{HH}^{n}_{\ell}(\Gamma,\DD\Gamma)=0$ for every $n\geq2$. In particular,
    \[
    \operatorname{HH}^{n}_{\ell}\bigl(T,T[2-n]\bigr)=0
    \qquad(n>2).
    \]
    Hence $T$ is intrinsically formal by \cite[Theorem~4.7]{SeidelThomas01}, and therefore $B\simeq T$.
\end{proof}

\begin{proof}[Proof of \cref{thm:negative-KR}]
    Let
    \[
        T:=kQ\ltimes \Sigma^{-d-1}\DD(kQ).
    \]
    By \cref{lem:KR-koszul-model}, $B\simeq T$. Moreover, the proof of \cref{cor:CM-characterization} applied to the realizing algebra $A$ gives
    \[
        \C\simeq\underline{\CM}(B)\simeq\underline{\CM}(T).
    \]
    By \cite[Theorem~2.4 and Proposition~6.3]{Jin20},
    \[
        \underline{\CM}(T)\simeq\sg(T)\simeq\C_d(kQ).
    \]
    Hence $\C\simeq\C_d(kQ)$.
\end{proof}

\begin{rmk}\label{rmk:d-minus-2}
    The assumption $d\leq-3$ is used only in \cref{lem:KR-heart} to deduce that $\Hcal_\Lcal$ is hereditary. Thus the same statement holds for $d=-2$ if $\Hcal_\Lcal$ is hereditary.
\end{rmk}

\begin{exa}[A cyclic family and the Calabi--Yau period]\label{exa:cyclic-family}
Put $q=-d\geq2$, and let $m\geq2$. Let $Q_m$ be the oriented cycle with vertices $\mathbb Z/m\mathbb Z$ and arrows
\[
\alpha_i:i\longrightarrow i+1
\qquad(i\in\mathbb Z/m\mathbb Z).
\]
Set
\[
A_m=kQ_m,
\qquad |\alpha_i|=q,
\]
with zero differential. Then $A_m$ is a locally finite strictly positive dg algebra with $A_m^0\simeq k^m$. Up to reversing the cyclic labelling, its Koszul dual is
\[
B_m:=A_m^!\simeq kQ_m/\mathfrak r^2,
\qquad |\beta_i|=1-q=d+1,
\]
with zero differential, where $\mathfrak r=(\beta_i)_i$ is the arrow ideal. If $\sigma(i)=i-1$, then
\begin{equation}\label{eq:cycle-nakayama}
\DD(B_me_i)\simeq\Sigma^{1-q}e_{\sigma(i)}B_m.
\end{equation}
Hence
\[
\add\DD B_m=\add\Sigma^{1-q}B_m=\add\Sigma^{d+1}B_m,
\]
so $B_m$ is $q$-self-injective and $A_m$ is $d$-stable by \cref{prop:d-stable-selfinjective}. Since the Nakayama permutation $\sigma$ has order $m>1$, \eqref{eq:cycle-nakayama} also shows that $B_m$ is not shifted symmetric.

We claim that
\begin{equation}\label{eq:cyclic-cosg}
\cosg(A_m)\simeq\C_{md}(k).
\end{equation}
Let $S_i$ be the simple right $B_m$-module at $i$, and write $\overline S_i$ for its image in $\C:=\sg(B_m)$. By \cite[Theorems~2.4 and~5.6]{Jin20}, the collection $\{\overline S_i\}_{i\in\mathbb Z/m\mathbb Z}$ is a $q$-simple-minded system in $\C$. The projective presentations
\[
0\longrightarrow\Sigma^{q-1}S_{i+1}\longrightarrow e_iB_m\longrightarrow S_i\longrightarrow0
\]
give
\begin{equation}\label{eq:cycle-shift-permutation}
\Sigma^q\overline S_i\simeq\overline S_{i-1}.
\end{equation}
Together with the $q$-simple-minded-system axioms, \eqref{eq:cycle-shift-permutation} gives
\[
\C=\add\{\Sigma^r\overline S_0\mid0\leq r<mq\},
\qquad
\Sigma^{mq}\overline S_0\simeq\overline S_0,
\]
and these $mq$ objects are pairwise orthogonal bricks. Indeed, by iterating \eqref{eq:cycle-shift-permutation}, any morphism between two objects $\Sigma^a\overline S_0$ and $\Sigma^b\overline S_0$, with $0\leq a,b<mq$, reduces to a morphism
\[
\overline S_i\longrightarrow\Sigma^t\overline S_j
\qquad(1-q\leq t\leq0).
\]
For $t<0$ this vanishes by the $q$-simple-minded-system vanishing, while for $t=0$ the Schur condition gives zero unless $i=j$, in which case the endomorphism space is $k$. Hence $\C$ is semisimple, its Serre functor is the identity, and $\Sigma^{mq}\simeq\id$. Since $mq=-md$, the category $\C$ is $md$-Calabi--Yau. Moreover, $\{\overline S_0\}$ is an $mq$-simple-minded system and
\[
\bigoplus_{r=0}^{mq-1}\Hom_{\C}(\overline S_0,\Sigma^r\overline S_0)\simeq k
\]
as a graded algebra. Since $md\leq-4$, \cref{thm:negative-KR}, applied with Calabi--Yau dimension $md$ and the quiver $A_1$, yields
\[
\sg(B_m)=\C\simeq\C_{md}(k).
\]
Finally $B_m^{\mathrm{o}}\simeq B_m$, and $k$-duality identifies $\C_{md}(k)^{\mathrm{o}}$ with $\C_{md}(k)$. Hence \cref{rmk:pvd-finite} gives \eqref{eq:cyclic-cosg}. Since $\Sigma^q\overline S_0\simeq\overline S_{-1}\not\simeq\overline S_0$, the identity Serre functor is not isomorphic to $\Sigma^d=\Sigma^{-q}$; thus $\cosg(A_m)$ is not $d$-Calabi--Yau. More generally, it is not $jd$-Calabi--Yau for $1\leq j<m$.
\end{exa}

\section{Normalized Calabi--Yau structures}\label{sec:cy}

We now study the localization step at the level of mixed complexes; \cref{sec:left-right-CY} treats the Koszul-dual step.

\noindent\textbf{Mixed complexes.}
Following Kassel \cite{Kassel87}, let $\Lambda=k[\epsilon]/(\epsilon^2)$, where $|\epsilon|=-1$ and $d\epsilon=0$, and recall that a mixed complex is a dg $\Lambda$-module. For a small dg $k$-category $\A$, we write $M(\A)$ for the mixed complex associated with $\A$ by Keller \cite{Keller99}. With our cohomological grading convention, we use homological indexing $H_i(C)=H^{-i}(C)$ and set
\[
\begin{aligned}
\HHom_i(\A)&=H_i\bigl(M(\A)\bigr),\\
\HChom_i(\A)&=H_i\bigl(M(\A)\otimes_{\Lambda}^{\mathbf L}k\bigr),\\
\HN_i(\A)&=H_i\RHom_{\Lambda}\bigl(k,M(\A)\bigr).
\end{aligned}
\]
The augmentation $\Lambda\to k$ induces the canonical maps
\[
\HN_i(\A)\longrightarrow\HHom_i(\A)\longrightarrow\HChom_i(\A).
\]
\medskip\noindent\textbf{Right Calabi--Yau structures.}
We write $\DD\HHom_i(\A)$ and $\DD\HChom_i(\A)$ for the $k$-linear duals. There is a canonical map
\[
\DD\HHom_{-m}(\A)\longrightarrow
\Hom_{\D(\A^e)}(\Sigma^m\A,\DD\A).
\]
A class $x\in\DD\HHom_{-m}(\A)$ is \defterm{non-degenerate} if its image under this map is an isomorphism. A \defterm{right $m$-Calabi--Yau structure} on $\A$ is a class in $\DD\HChom_{-m}(\A)$ whose image in $\DD\HHom_{-m}(\A)$, dual to the canonical map $\HHom_{-m}(\A)\to\HChom_{-m}(\A)$, is non-degenerate.

\medskip\noindent\textbf{Reduced mixed complexes.}
Throughout this section, let $A$ be a $d$-stable locally finite strictly positive dg algebra, and put $\ell=A^0\simeq k^r$, regarded as a dg algebra concentrated in degree zero. The inclusion $\ell\hookrightarrow A$ defines a dg functor
\[
j:\ell\longrightarrow\per_{\dg}(A),\qquad *\longmapsto A,
\]
and we also consider its composite with the quotient $q:\per_{\dg}(A)\to\cosg_{\dg}(A)$. Define the reduced mixed complexes
\[
\begin{aligned}
\overline M(\per_{\dg}(A);\ell)
&=\operatorname{Cone}\bigl(M(\ell)\to M(\per_{\dg}(A))\bigr),\\
\overline M(\cosg_{\dg}(A);\ell)
&=\operatorname{Cone}\bigl(M(\ell)\to M(\cosg_{\dg}(A))\bigr).
\end{aligned}
\]
For either reduced mixed complex $\overline M$, we write
\[
\overline{\HHom}_i=H_i(\overline M),
\qquad
\overline{\HChom}_i=H_i\bigl(\overline M\otimes_{\Lambda}^{\mathbf L}k\bigr),
\]
with the dg category and semisimple algebra displayed when needed.

\begin{prop}[Reduced localization]\label{prop:reduced-localization}
There is a distinguished triangle
\[
M(\pvd_{\dg}(A))\longrightarrow\overline M(\per_{\dg}(A);\ell)\longrightarrow\overline M(\cosg_{\dg}(A);\ell)\longrightarrow.
\]
\end{prop}

\begin{proof}
By Keller's localization theorem for mixed complexes \cite[Theorem~1.5]{Keller99},
\[
M(\pvd_{\dg}(A))\longrightarrow M(\per_{\dg}(A))\longrightarrow M(\cosg_{\dg}(A))\longrightarrow
\]
is a distinguished triangle. Taking the cones of the compatible morphisms from $M(\ell)$ gives the asserted triangle.
\end{proof}

\subsection{Normalized Hochschild and cyclic complexes}

Put $\overline A=\ker(A\to\ell)=A^{>0}$, and let $s$ denote the bar suspension of cohomological degree $-1$. The normalized Hochschild chains of Hochschild length $p$ are
\[
NC_p^{\ell}(A)=\left(A\otimes_{\ell}(s\overline A)^{\otimes_{\ell}p}\right)_{\ell},
\qquad
(-)_{\ell}=(-)\otimes_{\ell^e}\ell.
\]
Their direct sum $NC^{\ell}(A)=\bigoplus_{p\geq0}NC_p^{\ell}(A)$, equipped with the internal differential and the usual normalized Hochschild differential, is the normalized Hochschild complex. Here the superscript $\ell$ indicates that the normalized complex is taken over the separable subalgebra $\ell=A^0$. By Morita invariance of Keller's mixed complex and separability of $\ell$, the standard normalization over $\ell$ gives
\begin{equation}\label{eq:relative-HH-model}
H_i\bigl(NC^{\ell}(A)\bigr)\simeq\HHom_i(\per_{\dg}(A));
\end{equation}
see \cite{Keller99} and \cite[Chapter~1]{Loday98}.

Under the Morita identification $M(\per_{\dg}(A))\simeq M(A)$, the map from $M(\ell)$ is induced by the inclusion $\ell\hookrightarrow A$. On underlying Hochschild complexes, the copy of $\ell$ in Hochschild length zero is therefore a split subcomplex, and the cone $\overline M(\per_{\dg}(A);\ell)$ is represented by the quotient
\[
\overline{NC}^{\ell}(A)=NC^{\ell}(A)/\ell.
\]

For cyclic homology, it is more convenient to use the reduced cyclic $\lambda$-complex directly. In Hochschild length $p$, put
\[
\overline C_p^{\lambda,\ell}(A)
=
\frac{\left(\overline A\otimes_{\ell}(s\overline A)^{\otimes_{\ell}p}\right)_{\ell}}
{(1-\tau_p)\left(\overline A\otimes_{\ell}(s\overline A)^{\otimes_{\ell}p}\right)_{\ell}},
\]
where $\tau_p$ is the signed cyclic permutation with the standard Koszul sign convention. Since every tensor factor lies in the augmentation ideal $\overline A$, the operator $\tau_p$ preserves the displayed space and has cohomological degree zero. Let
\[
\overline C^{\lambda,\ell}(A)=\bigoplus_{p\geq0}\overline C_p^{\lambda,\ell}(A)
\]
with the differential induced by the internal and Hochschild differentials. Since $\operatorname{char}k=0$ and $\ell$ is separable, relative normalization identifies the cyclic homology of $\overline M(\per_{\dg}(A);\ell)$ with the homology of this relative $\lambda$-complex:
\begin{equation}\label{eq:relative-HC-model}
H_i\bigl(\overline C^{\lambda,\ell}(A)\bigr)
\simeq
\overline{\HChom}_i(\per_{\dg}(A);\ell).
\end{equation}
This is the separable-base relative analogue of \cite[Proposition~2.2.14]{Loday98}; compare the relative mixed-complex construction in \cite[Section~4.3]{KellerLiuAmiot24}.

For a homogeneous normalized Hochschild chain $a_0[sa_1|\cdots|sa_p]$, strict positivity gives
\begin{equation}\label{eq:normalized-degree}
|a_0|+\sum_{j=1}^p(|a_j|-1)\geq0.
\end{equation}
Hence $NC^{\ell}(A)$ is concentrated in non-negative cohomological degrees. In the reduced cyclic complex one also has $a_0\in\overline A$, so the same degree is in fact at least $1$; therefore $\overline C^{\lambda,\ell}(A)$ is concentrated in positive cohomological degrees.

Keller's Hochschild localization sequence for
$\pvd_{\dg}(A)\to\per_{\dg}(A)\to\cosg_{\dg}(A)$ and the reduced cyclic localization sequence from \cref{prop:reduced-localization} contain boundary morphisms
\begin{align}
\partial_{\mathrm{HH}}&:\HHom_{-d}(\cosg_{\dg}(A))\longrightarrow\HHom_{-d-1}(\pvd_{\dg}(A)),\label{eq:HH-boundary}\\
\delta_{\mathrm{red}}&:\overline{\HChom}_{-d}(\cosg_{\dg}(A);\ell)\longrightarrow\HChom_{-d-1}(\pvd_{\dg}(A)).\label{eq:reduced-boundary}
\end{align}

\begin{prop}\label{prop:reduced-vanishing}
For every $m>0$,
\[
\HHom_m(\per_{\dg}(A))=0,
\qquad
\overline{\HChom}_m(\per_{\dg}(A);\ell)=0.
\]
In particular, both boundary morphisms \eqref{eq:HH-boundary} and \eqref{eq:reduced-boundary} are isomorphisms.
\end{prop}

\begin{proof}
The degree estimates above, together with \eqref{eq:relative-HH-model} and \eqref{eq:relative-HC-model}, give
\[
\HHom_m(\per_{\dg}(A))=0,
\qquad
\overline{\HChom}_m(\per_{\dg}(A);\ell)=0
\qquad(m>0).
\]
Since $-d$ and $-d-1$ are positive, the Hochschild and reduced cyclic localization long exact sequences show that \eqref{eq:HH-boundary} and \eqref{eq:reduced-boundary} are isomorphisms.
\end{proof}

Since $\ell$ is semisimple, $\HHom_i(\ell)=0$ for $i\neq0$. Hence, for $d\leq-2$,
\begin{equation}\label{eq:reduced-HH-identification}
\HHom_{-d}(\cosg_{\dg}(A))\xrightarrow{\sim}\overline{\HHom}_{-d}(\cosg_{\dg}(A);\ell).
\end{equation}

\subsection{Non-degeneracy}

For $X,Y\in\per(A)$, a morphism $u:Y\to E$ with $E\in\pvd(A)$ is called a \defterm{local $\pvd(A)$-envelope of $Y$ relative to $X$} if postcomposition with $u$ induces an injection
\[
    \Hom_{\D(A)}(X,Y)\lhook\joinrel\longrightarrow\Hom_{\D(A)}(X,E).
\]
After passing to opposite categories, this is the local-cover condition of \cite[Definition~1.2]{Amiot09} used in \cite[Lemma~4.3.4]{HaniharaLiu26}: a local envelope $Y\to E$ relative to $X$ becomes a local cover $E^{\mathrm{o}}\to Y^{\mathrm{o}}$ relative to $X^{\mathrm{o}}$.

We use the canonical weight structure on $\D(A)$ from \cite[Corollary~5.1]{KN13}. There, the two halves are characterized by cohomological vanishing, and the associated truncation triangles induce isomorphisms in the complementary cohomological ranges. For $m\in\mathbb Z$, choose such a triangle
\begin{equation}\label{eq:positive-weight-truncation}
    \sigma_{\geq m}Y\longrightarrow Y\longrightarrow\sigma_{<m}Y\longrightarrow\Sigma\sigma_{\geq m}Y,
\end{equation}
such that
\[
    H^i(\sigma_{\geq m}Y)=0\quad(i<m),
    \qquad
    H^i(\sigma_{<m}Y)=0\quad(i\geq m),
\]
and the two maps induce the corresponding isomorphisms on the complementary cohomological ranges.

\begin{lem}\label{lem:local-envelope}
For $X,Y\in\per(A)$, the object $Y$ admits a local $\pvd(A)$-envelope relative to $X$. Consequently,
\[
\per_{\dg}(A)\xrightarrow{\sim}
\RHom_{\pvd_{\dg}(A)}\bigl(\per_{\dg}(A),\per_{\dg}(A)\bigr)
\]
in $\D(\per_{\dg}(A)^e)$.
\end{lem}

\begin{proof}
Choose a weight truncation \eqref{eq:positive-weight-truncation}. Since $Y$ is perfect over the strictly positive locally finite dg algebra $A$, its cohomology is locally finite and bounded below. Thus, for every $m$, the object $\sigma_{<m}Y$ has finite-dimensional total cohomology. Hence
\[
    \sigma_{<m}Y\in\pvd(A)\subseteq\per(A)
\]
by \eqref{eq:pvd-thick}.

For the fixed perfect object $X$, choose a finite construction of $X$ from shifts of direct summands of $A$. Since $H^{<m}(\sigma_{\geq m}Y)=0$, this finite construction shows that, for $m\gg0$,
\[
    \Hom_{\D(A)}(X,\sigma_{\geq m}Y)=0
    =\Hom_{\D(A)}(X,\Sigma\sigma_{\geq m}Y).
\]
Applying $\Hom_{\D(A)}(X,-)$ to \eqref{eq:positive-weight-truncation} therefore gives an isomorphism
\[
    \Hom_{\D(A)}(X,Y)\xrightarrow{\sim}\Hom_{\D(A)}(X,\sigma_{<m}Y),
\]
so $Y\to\sigma_{<m}Y$ is the required local envelope; in fact, it induces an isomorphism on the relevant $\Hom$ spaces. Passing to opposite dg categories turns these local envelopes into the local covers of \cite[Lemma~4.3.4]{HaniharaLiu26}. Hence that lemma gives
\[
\per_{\dg}(A)^{\mathrm{o}}\xrightarrow{\sim}
\RHom_{\pvd_{\dg}(A)^{\mathrm{o}}}\bigl(\per_{\dg}(A)^{\mathrm{o}},\per_{\dg}(A)^{\mathrm{o}}\bigr).
\]
Passing back to opposites yields the asserted isomorphism in $\D(\per_{\dg}(A)^e)$.
\end{proof}

\begin{lem}\label{lem:localization-amplitude}
Let
\[
K=\per_{\dg}(A)\otimes^{\mathbf L}_{\pvd_{\dg}(A)}\per_{\dg}(A).
\]
For $P,Q\in\add A$,
\[
K(P,Q)\in \D^{\leq d+1}(k).
\]
\end{lem}

\begin{proof}
The map
\[
\Hom_{\D(A)}(P,\Sigma^iQ)\longrightarrow\Hom_{\cosg(A)}(qP,\Sigma^iqQ)
\]
is an isomorphism for $i>0$ by \cref{prop:fundamental-domain}. For $i=0$, \cref{prop:fundamental-domain}(1) identifies the target with the source modulo morphisms factoring through $\add S_A$. Every such factorization is zero because
\[
    \Hom_{\D(A)}(S_A,Q)=H^0\RHom_A(S_A,Q)=0
\]
by $d$-stability, so the degree-zero map is also an isomorphism. Finally, for $d+1\leq i<0$, the source vanishes by strict positivity, while the target vanishes by \cref{prop:fundamental-domain}(1), since $P,\Sigma^iQ\in\F_A$. Evaluating the dg localization triangle
\[
K\longrightarrow\per_{\dg}(A)\xrightarrow{q}\cosg_{\dg}(A)\longrightarrow\Sigma K
\]
at $P,Q$ gives the assertion.
\end{proof}

\begin{prop}\label{prop:nondegeneracy}
The dual Hochschild boundary
\[
\delta_{\mathrm{HH}}:=\partial_{\mathrm{HH}}^\vee:
\DD\HHom_{-d-1}(\pvd_{\dg}(A))\xrightarrow{\sim}\DD\HHom_{-d}(\cosg_{\dg}(A))
\]
is an isomorphism and preserves and detects non-degeneracy.
\end{prop}

\begin{proof}
The isomorphism assertion follows from \cref{prop:reduced-vanishing}. It remains to prove the non-degeneracy statement. For brevity in this proof only, write
\[
\Pdg=\pvd_{\dg}(A),\qquad \Tdg=\per_{\dg}(A),\qquad \Qdg=\cosg_{\dg}(A).
\]

\emph{Preservation.}
We apply \cite[Theorem~4.3.1]{HaniharaLiu26} to the exact sequence
\[
\Pdg\longrightarrow\Tdg\longrightarrow\Qdg,
\]
with the autoequivalence $\nu=\Sigma^{-d-1}$ of $\Tdg$ and the bimodule
\[
M=\Sigma^{d+1}\Tdg\simeq\Tdg(\nu ?,-).
\]
The autoequivalence $\nu$ preserves $\Pdg$, and the induced coefficient bimodules are
\[
M_{\Pdg}\simeq\Sigma^{d+1}\Pdg,
\qquad
M_{\Qdg}\simeq\Sigma^{d+1}\Qdg.
\]
Thus localization lowers the Calabi--Yau degree by one: a class on $\Pdg$ is represented by $\Sigma^{d+1}\Pdg\to\DD\Pdg$, whereas its boundary on $\Qdg$ is represented by $\Sigma^d\Qdg\to\DD\Qdg$. Hence the connecting morphism of \cite[Theorem~4.3.1]{HaniharaLiu26} is precisely $\delta_{\mathrm{HH}}$. Since $A$ is locally finite, every morphism complex of $\Tdg=\per_{\dg}(A)$ has degreewise finite-dimensional cohomology. Hence the canonical biduality morphism
\[
\Tdg\longrightarrow\DD\DD\Tdg
\]
is an isomorphism in $\D(\Tdg^e)$. On the other hand, \cref{lem:local-envelope} gives
\[
\Tdg\xrightarrow{\sim}\RHom_{\Pdg}(\Tdg,\Tdg).
\]
\cite[Theorem~4.3.1(b)]{HaniharaLiu26} therefore applies, and $\delta_{\mathrm{HH}}$ preserves non-degeneracy.

\emph{Detection.}
Let $x\in\DD\HHom_{-d-1}(\Pdg)$ and assume that $\delta_{\mathrm{HH}}(x)$ is non-degenerate. Let
\[
    f_x:\Sigma^{d+1}\Pdg\longrightarrow \DD\Pdg,
    \qquad
    g_x:\Sigma^d\Qdg\longrightarrow \DD\Qdg
\]
be the bimodule morphisms corresponding to $x$ and $\delta_{\mathrm{HH}}(x)$, respectively. Put
\[
    K=\per_{\dg}(A)\otimes^{\mathbf L}_{\pvd_{\dg}(A)}\per_{\dg}(A).
\]
Using the adjunctions of \cite[Lemma~4.2.1]{HaniharaLiu26}, let
\[
    u_x:\Sigma^{d+1}K\longrightarrow\DD\Tdg,
    \qquad
    v_x:\Sigma^{d+1}\Tdg\longrightarrow\DD K
\]
be the morphisms corresponding to $f_x$. By \cite[Proposition~4.2.2 and Theorem~4.3.1(a)]{HaniharaLiu26}, they fit into a morphism of localization triangles in $\D(\Tdg^e)$:
\[
\begin{tikzcd}[column sep=large]
\Sigma^{d+1}K \ar[r] \ar[d,"u_x"'] & \Sigma^{d+1}\Tdg \ar[r] \ar[d,"v_x"'] & \Sigma^{d+1}\Qdg \ar[r] \ar[d,"\Sigma g_x"'] & {} \\
\DD\Tdg \ar[r] & \DD K \ar[r] & \Sigma \DD\Qdg \ar[r] & {}.
\end{tikzcd}
\]
Here the third vertical morphism denotes the restriction of $\Sigma g_x$ along $\Tdg^e\to\Qdg^e$. Since $g_x$ is an isomorphism, the morphism of triangles yields
\[
    \operatorname{Cone}(u_x)\simeq\operatorname{Cone}(v_x).
\]
For $P,Q\in\add A$, strict positivity and \cref{lem:localization-amplitude} give
\[
    \operatorname{Cone}(u_x)(P,Q)\in \D^{\leq0}(k),
    \qquad
    \operatorname{Cone}(v_x)(P,Q)\in \D^{\geq-d-2}(k).
\]

\emph{The case $d\leq-3$.}
Here $-d-2\geq1$, so the two bounds are disjoint. Since the two cones are isomorphic, both vanish.

\emph{The case $d=-2$.}
The two bounds show that the common cone can have cohomology only in degree zero. We claim that $H^0(u_x(P,Q))$ is an isomorphism. Since $\Tdg(P,Q)$ is concentrated in non-negative degrees, the localization triangle gives a boundary isomorphism
\[
    \partial:H^{-2}\Qdg(P,Q)\xrightarrow{\sim}H^{-1}K(P,Q).
\]
Moreover, the degree-zero map
\[
    H^0\Tdg(Q,P)\xrightarrow{\sim}H^0\Qdg(Q,P)
\]
is an isomorphism by the proof of \cref{lem:localization-amplitude}. Evaluating the morphism of localization triangles at $(P,Q)$ gives a morphism of long exact cohomology sequences. By naturality, the resulting square commutes up to the standard sign of the connecting morphism:
\[
\begin{tikzcd}[column sep=large,row sep=large]
H^{-1}K(P,Q) \ar[r,"{H^0(u_x(P,Q))}"] \ar[d,"\partial^{-1}"'] & \DD H^0\Tdg(Q,P) \ar[d,"{(\DD H^0(q_{Q,P}))^{-1}}"] \\
H^{-2}\Qdg(P,Q) \ar[r,"{\langle-,-\rangle_{g_x}}"'] & \DD H^0\Qdg(Q,P).
\end{tikzcd}
\]
The lower horizontal arrow is an isomorphism because $g_x$ is non-degenerate, and the two vertical arrows are isomorphisms by the preceding paragraph. Hence $H^0(u_x(P,Q))$ is an isomorphism. Since the source of $u_x(P,Q)$ lies in $\D^{\leq0}(k)$, the long exact cohomology sequence of its cone, together with the preceding amplitude bounds, shows that the common cone vanishes.

Thus $u_x$ is an isomorphism on $\add A$. Since $\Tdg=\thick(A)$, it is an isomorphism in $\D(\Tdg^e)$. By the adjunctions of \cite[Lemma~4.2.1]{HaniharaLiu26}, restriction along $\Pdg^e\to\Tdg^e$ identifies
\[
    K|_{\Pdg^e}\simeq\Pdg,
    \qquad
    (\DD\Tdg)|_{\Pdg^e}\simeq \DD\Pdg,
\]
and, under these identifications, $u_x|_{\Pdg^e}$ is exactly $f_x$. Therefore $f_x$ is an isomorphism and $x$ is non-degenerate.
\end{proof}

\subsection{Normalized Calabi--Yau structures}

\begin{dfn}\label{dfn:normalized-CY}
An \defterm{$\ell$-normalized right $d$-Calabi--Yau structure} on $\cosg_{\dg}(A)$ is a class
\[
\alpha\in \DD\overline{\HChom}_{-d}(\cosg_{\dg}(A);\ell)
\]
whose image in $\DD\HHom_{-d}(\cosg_{\dg}(A))$ via \eqref{eq:reduced-HH-identification} is non-degenerate.
\end{dfn}

This should not be confused with a relative Calabi--Yau structure on $\ell\to\cosg_{\dg}(A)$: here non-degeneracy is imposed only on the induced absolute Hochschild class of $\cosg_{\dg}(A)$.

\begin{thm}[Hochschild localization and normalized Calabi--Yau correspondence]\label{thm:normalized-CY}
Let $d\leq-2$, and let $A$ be a $d$-stable locally finite strictly positive dg algebra with $\ell=A^0\simeq k^r$. Then the Hochschild boundary \eqref{eq:HH-boundary} is an isomorphism, and its dual preserves and detects non-degeneracy. Moreover, the dual of the reduced cyclic boundary \eqref{eq:reduced-boundary} induces a bijection
\[
\begin{gathered}
\left\{\text{right $(d+1)$-Calabi--Yau structures on $\pvd_{\dg}(A)$}\right\}\\[-1mm]
\xrightarrow{\sim}\\[-1mm]
\left\{\text{$\ell$-normalized right $d$-Calabi--Yau structures on $\cosg_{\dg}(A)$}\right\}.
\end{gathered}
\]
\end{thm}

\begin{proof}
The Hochschild assertion is \cref{prop:nondegeneracy}. Naturality of the passage from cyclic to Hochschild homology gives a commutative square
\[
\begin{tikzcd}[column sep=large,row sep=large]
\DD\HChom_{-d-1}(\pvd_{\dg}(A)) \ar[r,"\delta_{\mathrm{red}}^\vee"] \ar[d] & \DD\overline{\HChom}_{-d}(\cosg_{\dg}(A);\ell) \ar[d] \\
\DD\HHom_{-d-1}(\pvd_{\dg}(A)) \ar[r,"\delta_{\mathrm{HH}}"'] & \DD\HHom_{-d}(\cosg_{\dg}(A)),
\end{tikzcd}
\]
where the right vertical arrow is the cyclic-to-Hochschild map followed by the dual of \eqref{eq:reduced-HH-identification}. The upper horizontal arrow is an isomorphism by \cref{prop:reduced-vanishing}, while the lower horizontal arrow preserves and detects non-degeneracy by \cref{prop:nondegeneracy}. Therefore the upper isomorphism restricts to a bijection between the two sets of Calabi--Yau structures in the statement.
\end{proof}

For the realizing algebra of \cref{thm:realization}, the primitive idempotents $e_i\in\ell=A^0$ are labelled so that the quasi-equivalence $\overline\pi_{\dg}:\cosg_{\dg}(A)\simeq\C_{\dg}$ sends $e_iA$ to $L_i$.

\subsection{Ordinary Calabi--Yau structures}

The relation with ordinary cyclic classes is governed by the exact sequence
\begin{equation}\label{eq:ordinary-reduced-sequence}
\begin{aligned}
\HChom_{-d}(\ell)&\longrightarrow\HChom_{-d}(\cosg_{\dg}(A))\longrightarrow\overline{\HChom}_{-d}(\cosg_{\dg}(A);\ell)\\
&\longrightarrow\HChom_{-d-1}(\ell)\longrightarrow\HChom_{-d-1}(\cosg_{\dg}(A)),
\end{aligned}
\end{equation}
together with
\begin{equation}\label{eq:semisimple-cyclic}
\HChom_{2n}(\ell)\simeq k^r,\qquad \HChom_{2n+1}(\ell)=0\qquad(n\geq0).
\end{equation}

\begin{prop}\label{prop:parity}
\begin{enumerate}[label=\textup{(\arabic*)}]
\item If $d$ is even, the forgetful map
\[
\DD\overline{\HChom}_{-d}(\cosg_{\dg}(A);\ell)\longrightarrow\DD\HChom_{-d}(\cosg_{\dg}(A))
\]
is injective. Its image consists of the functionals annihilating the image of $\HChom_{-d}(\ell)\to\HChom_{-d}(\cosg_{\dg}(A))$.
\item If $d$ is odd and
\[
N=\ker\bigl(\HChom_{-d-1}(\ell)\to\HChom_{-d-1}(\cosg_{\dg}(A))\bigr),
\]
then
\[
0\longrightarrow \DD N\longrightarrow \DD\overline{\HChom}_{-d}(\cosg_{\dg}(A);\ell)\longrightarrow\DD\HChom_{-d}(\cosg_{\dg}(A))\longrightarrow0.
\]
Hence every ordinary right $d$-Calabi--Yau structure admits a normalized lift.
\end{enumerate}
\end{prop}

\begin{proof}
Both statements follow immediately from \eqref{eq:ordinary-reduced-sequence} and \eqref{eq:semisimple-cyclic}. The underlying Hochschild class is unchanged under \eqref{eq:reduced-HH-identification}.
\end{proof}

In particular, if $d$ is odd and $\cosg_{\dg}(A)$ has a right $d$-Calabi--Yau structure, then $\pvd_{\dg}(A)$ has a right $(d+1)$-Calabi--Yau structure by \cref{thm:normalized-CY,prop:parity}.

\section{Left--right Calabi--Yau symmetry}\label{sec:left-right-CY}

Throughout this section, let $A$ be a locally finite strictly positive dg algebra with
\[
\ell=A^0\simeq k^r,\qquad B=A^!
\]
and assume that $B$ is proper. By \cref{prop:smooth-proper}, the dg category $\per_{\dg}(A)$ is smooth and $\pvd_{\dg}(A)$ is proper.

For a smooth small dg category $\A$, put
\[
\A^{\vee}:=\RHom_{\A^e}(\A,\A^e),
\]
viewed as the inverse dualizing bimodule. Following \cite{BravDyckerhoff19}, a \defterm{left $n$-Calabi--Yau structure} on $\A$ is a class in $\HN_n(\A)$ whose Hochschild image corresponds to an isomorphism
\[
\Sigma^n\A^{\vee}\xrightarrow{\sim}\A
\]
in $\D(\A^e)$.

If $P\in\pvd(A)$ and $X\in\per(A)$, then $\RHom_A(X,P)$ is perfect over $k$, since $P$ has finite-dimensional total cohomology and $X$ is finitely built from $A$. Hence $\pvd_{\dg}(A)\subseteq\per_{\dg}(A)$ consists of locally perfect objects, and Brav--Dyckerhoff construct an $S^1$-equivariant evaluation morphism, equivalently a morphism of mixed complexes,
\begin{equation}\label{eq:BD-evaluation}
\Theta:M(\per_{\dg}(A))\longrightarrow \DD M(\pvd_{\dg}(A)).
\end{equation}
The key point is that the bimodule duality underlying \eqref{eq:BD-evaluation} is ordinary Koszul duality for the enveloping algebra.

\begin{lem}\label{lem:enveloping-Morita-duality}
Let
\[
\Phi_e=\RHom_{A^e}(-,\ell^e).
\]
Under the Morita identifications $\D(\per_{\dg}(A)^e)\simeq \D(A^e)$ and the dg Koszul equivalence
\[
\pvd_{\dg}(A)\simeq\per_{\dg}(B^{\mathrm{o}})^{\mathrm{o}},
\]
the Brav--Dyckerhoff bimodule duality functor underlying \eqref{eq:BD-evaluation}, restricted to $\per(A^e)$, identifies with the Koszul dual functor $\Phi_e$ after the canonical graded flip $(B^e)^{\mathrm{o}}\simeq B^e$. In particular, it is quasi-fully faithful on $\per(A^e)$.
\end{lem}

\begin{proof}
For brevity in this proof only, write
\[
\Tdg=\per_{\dg}(A),\qquad \Pdg=\pvd_{\dg}(A).
\]
By \eqref{eq:pvd-thick}, the dg category $\Pdg$ is split-generated by the summands of $S_A=\ell$. Let $\mathfrak D_{\mathrm{BD}}$ denote the contravariant bimodule functor underlying the Brav--Dyckerhoff evaluation morphism; this is the functor inducing the lower horizontal morphism in the square of \cite[Theorem~3.1]{BravDyckerhoff19}. For a $\Tdg$-bimodule $\mathcal M$, set
\[
\mathfrak D_{\mathrm{BD}}(\mathcal M)(P,Q)
=
\RHom_{\Tdg}\bigl(\mathcal M(-,Q),\Tdg(-,P)\bigr)
\qquad(P,Q\in\Pdg).
\]
Under the Morita equivalence $\D(\Tdg^e)\simeq \D(A^e)$, an $A$-bimodule $M$ corresponds to the $\Tdg$-bimodule
\[
\widetilde M(X,Y)
=
\RHom_A\bigl(X,Y\otimes_A^{\mathbf L}M\bigr).
\]
Hence dg Yoneda gives, for $P,Q\in\Pdg$,
\[
\mathfrak D_{\mathrm{BD}}(\widetilde M)(P,Q)
\simeq
\RHom_A\bigl(Q\otimes_A^{\mathbf L}M,P\bigr).
\]
Evaluating at the split-generator $\ell$ in both variables and using tensor--Hom adjunction, we obtain
\[
\begin{aligned}
\mathfrak D_{\mathrm{BD}}(\widetilde M)(\ell,\ell)
&\simeq
\RHom_A\bigl(\ell\otimes_A^{\mathbf L}M,\ell\bigr)\\
&\simeq
\RHom_{A^e}\bigl(M,\RHom_k(\ell,\ell)\bigr)\\
&\simeq
\RHom_{A^e}(M,\ell^e)
=
\Phi_e(M),
\end{aligned}
\]
where the third isomorphism is the identification $\RHom_k(\ell,\ell)\simeq\ell^e$ from the proof of \cref{prop:smooth-proper}.

The object $(\ell,\ell)$ split-generates $\Pdg^e$, and its endomorphism dg algebra is
\[
\REnd_{\Pdg^e}((\ell,\ell))
\simeq
B^{\mathrm{o}}\otimes_k B
=
B^e.
\]
On the other hand, the K\"unneth formula gives
\[
(A^e)^!\simeq B^e.
\]
Evaluation at the split-generator $(\ell,\ell)$ realizes the Morita equivalence $\D(\Pdg^e)\simeq\D(B^e)$, and the adjunction isomorphisms above are natural in $M$ and intertwine the induced $B^e$-actions. Thus, after the canonical graded flip $(B^e)^{\mathrm{o}}\simeq B^e$, $\mathfrak D_{\mathrm{BD}}$ restricted to $\per(A^e)$ is naturally identified with the Koszul dual functor $\Phi_e$.

Under this identification,
\[
\mathfrak D_{\mathrm{BD}}(\Tdg)\simeq\Pdg,
\qquad
\mathfrak D_{\mathrm{BD}}(\Tdg^{\vee})\simeq\DD\Pdg.
\]
Indeed, the diagonal $\Tdg$ corresponds to $A$, and the computation in \cref{prop:smooth-proper} gives $\Phi_e(A)\simeq B$. Moreover, since $A\in\per(A^e)$, its inverse dualizing bimodule $A^{\vee}=\RHom_{A^e}(A,A^e)$ is perfect, and dualizability gives
\[
\begin{aligned}
\Phi_e(A^{\vee})
&=\RHom_{A^e}(A^{\vee},\ell^e)\simeq A\otimes_{A^e}^{\mathbf L}\ell^e\\
&\simeq\DD\RHom_{A^e}(A,\ell^e)\simeq\DD B.
\end{aligned}
\]
Finally, $A^e$ is pvd-finite by the proof of \cref{prop:smooth-proper}. Therefore \cite[Theorem~D(2)]{Fus24} makes $\Phi_e$ fully faithful on $\per(A^e)$. Applying this to all shifts shows that the induced morphisms of dg mapping complexes are quasi-isomorphisms. Hence $\Phi_e$, and therefore $\mathfrak D_{\mathrm{BD}}$, is quasi-fully faithful on $\per(A^e)$.
\end{proof}

\begin{thm}[Mixed-complex duality and left--right Calabi--Yau symmetry]\label{thm:left-right-CY}
Let $A$ be a locally finite strictly positive dg algebra with $A^0\simeq k^r$ and proper Koszul dual $B=A^!$. Then the Brav--Dyckerhoff evaluation morphism
\[
\Theta:M(\per_{\dg}(A))\longrightarrow\DD M(\pvd_{\dg}(A))
\]
is a quasi-isomorphism of mixed complexes. Consequently, for every $n\in\mathbb Z$ it induces canonical isomorphisms
\begin{equation}\label{eq:left-right-HH}
\Theta_n^{\mathrm{HH}}:\HHom_n(\per_{\dg}(A))\xrightarrow{\sim}\DD\HHom_{-n}(\pvd_{\dg}(A))
\end{equation}
and
\begin{equation}\label{eq:left-right-HN-HC}
\Theta_n:\HN_n(\per_{\dg}(A))\xrightarrow{\sim}\DD\HChom_{-n}(\pvd_{\dg}(A)).
\end{equation}
Under the second isomorphism, left $n$-Calabi--Yau structures on $\per_{\dg}(A)$ correspond bijectively to right $n$-Calabi--Yau structures on $\pvd_{\dg}(A)$.
\end{thm}

\begin{proof}
For brevity in this proof only, write
\[
\Tdg=\per_{\dg}(A),\qquad \Pdg=\pvd_{\dg}(A).
\]
Write $C_{\mathrm{HH}}(\A)$ for the underlying Hochschild complex of the mixed complex $M(\A)$. On underlying Hochschild complexes, the Brav--Dyckerhoff morphism fits into the commutative square of \cite[Theorem~3.1]{BravDyckerhoff19}
\[
\begin{tikzcd}[column sep=large,row sep=large]
C_{\mathrm{HH}}(\Tdg) \ar[r,"\Theta_{\mathrm{HH}}"] \ar[d,"\sim"'] & \DD C_{\mathrm{HH}}(\Pdg) \ar[d,"\sim"] \\
\RHom_{\Tdg^e}(\Tdg^{\vee},\Tdg) \ar[r] & \RHom_{\Pdg^e}(\Pdg,\DD\Pdg),
\end{tikzcd}
\]
where the lower horizontal morphism is induced by the bimodule duality functor appearing in \cref{lem:enveloping-Morita-duality}. Since $\Tdg$ is smooth, both $\Tdg^{\vee}$ and $\Tdg$ correspond to perfect $A^e$-modules. By \cref{lem:enveloping-Morita-duality}, the bimodule duality sends them to $\DD\Pdg$ and $\Pdg$, respectively, and is quasi-fully faithful on these objects. Hence the lower horizontal morphism and therefore $\Theta_{\mathrm{HH}}$ are quasi-isomorphisms. Thus $\Theta$, already a morphism of mixed complexes, is a quasi-isomorphism. This proves \eqref{eq:left-right-HH}.

Applying $\RHom_{\Lambda}(k,-)$ to \eqref{eq:BD-evaluation} and using the definitions from \cref{sec:cy} together with the standard tensor--dual adjunction yields
\[
\HN_n(\Tdg)
\xrightarrow{\sim}
H_n\RHom_{\Lambda}(k,\DD M(\Pdg))
\simeq
\DD\HChom_{-n}(\Pdg),
\]
which is \eqref{eq:left-right-HN-HC}.

It remains to compare non-degeneracy. Let $\xi\in\HN_n(\Tdg)$ and let
\[
\varphi_{\xi}:\Sigma^n\Tdg^{\vee}\longrightarrow\Tdg
\]
be the bimodule morphism determined by its Hochschild image. By the construction of the Brav--Dyckerhoff square and the definition of $\mathfrak D_{\mathrm{BD}}$ in \cref{lem:enveloping-Morita-duality}, the bimodule morphism associated with $\Theta_n(\xi)$ is obtained by applying this contravariant bimodule duality functor to $\varphi_{\xi}$. Since this functor reverses shifts, it sends
\[
\varphi_{\xi}:\Sigma^n\Tdg^{\vee}\longrightarrow\Tdg
\]
to a morphism
\[
\Pdg\longrightarrow \Sigma^{-n}\DD\Pdg,
\]
or equivalently
\[
\Sigma^n\Pdg\longrightarrow \DD\Pdg,
\]
which is exactly the bimodule morphism defining the right $n$-Calabi--Yau class $\Theta_n(\xi)$. The corresponding computation in the pseudocompact setting appears in \cite[Proposition~4.4.1]{KellerLiuAmiot24}. By \cref{lem:enveloping-Morita-duality}, the duality functor is quasi-fully faithful on the perfect $A^e$-modules containing the source and target of $\varphi_{\xi}$, and hence reflects isomorphisms there. Consequently,
\[
\varphi_{\xi}\text{ is an isomorphism}
\quad\Longleftrightarrow\quad
\Theta_n(\xi)\text{ is non-degenerate}.
\]
This proves the asserted bijection of Calabi--Yau structures.
\end{proof}

\begin{rmk}\label{rmk:pseudocompact-comparison}
Keller--Liu prove an analogous left--right correspondence using pseudocompact Koszul duality \cite[Proposition~4.4.1]{KellerLiuAmiot24}, while Holstein--Rivera obtain related results in a coalgebraic setting \cite{HolsteinRivera24}. Here no topological or coalgebraic replacement is needed: properness of the ordinary Koszul dual implies smoothness of $A$, and ordinary Koszul duality for $A^e$ identifies the Brav--Dyckerhoff evaluation morphism. In particular, \cref{thm:left-right-CY} is independent of the $d$-stability hypothesis used in \cref{sec:cy}.
\end{rmk}

Combining \cref{thm:left-right-CY} with the normalized localization correspondence of \cref{thm:normalized-CY} gives the following three-way symmetry.

\begin{cor}\label{cor:three-way-CY}
Let $d\leq-2$, and let $A$ be a $d$-stable locally finite strictly positive dg algebra with $\ell=A^0\simeq k^r$. Then there are canonical bijections
\[
\begin{gathered}
\left\{\begin{array}{c}\text{left $(d+1)$-Calabi--Yau structures}\\[-1mm]\text{on $\per_{\dg}(A)$}\end{array}\right\}
\xrightarrow{\sim}
\left\{\begin{array}{c}\text{right $(d+1)$-Calabi--Yau structures}\\[-1mm]\text{on $\pvd_{\dg}(A)$}\end{array}\right\}\\
\xrightarrow{\sim}
\left\{\begin{array}{c}\text{$\ell$-normalized right $d$-Calabi--Yau structures}\\[-1mm]\text{on $\cosg_{\dg}(A)$}\end{array}\right\}.
\end{gathered}
\]
\end{cor}

\defbibheading{bibliography}[\refname]{\section*{#1}}
\printbibliography

\end{document}